\documentclass{conm-p-l}

\usepackage{amsmath,amssymb,amsthm,slashed,amsfonts,verbatim}
\usepackage[top = 1in, left = 1.25in, right = 1in, bottom = 1in]{geometry}
\usepackage{cite}
\usepackage{graphicx}
\usepackage[all]{xy}
\usepackage{float}
\usepackage{tikz}
\usetikzlibrary{arrows,snakes}
\usepackage{enumerate}
\usepackage{indentfirst}

\newtheorem{theorem}{Theorem}[section]
\newtheorem{lemma}[theorem]{Lemma}
\newtheorem{prop}{Proposition}[theorem]
\newtheorem{cor}[theorem]{Corollary}

\theoremstyle{definition}
\newtheorem{definition}[theorem]{Definition}
\newtheorem{example}[theorem]{Example}
\newtheorem*{ack*}{Acknowledgment}

\theoremstyle{remark}
\newtheorem{remark}[theorem]{Remark}

\numberwithin{equation}{section}

\newcommand\ca{{\mathcal A}}

\newcommand\cc{{\mathcal C}}
\newcommand\cd{{\mathcal D}}

\newcommand\cg{{\mathcal G}}

\newcommand\cl{{\mathcal L}}

\newcommand\co{{\mathcal O}}

\newcommand\Bz{\mathbb Z}
\newcommand\Bn{\mathbb N}

\newcommand{\sig}{\text{sig}}
\newcommand{\fol}{\text{fol}}

\newcommand\natext{X^{+-}}%  {\tilde{X}}

\begin{document}

\title[Markov diagrams]{Markov Diagrams for Some Non-Markovian Systems}
\author{Kathleen Carroll}
\address{Department of Mathematics,
CB 3250 Phillips Hall,
University of North Carolina,
Chapel Hill, NC 27599 USA}
\curraddr{Department of Prosthetics and Orthotics, University of Hartford,
200 Bloomfield Avenue,
West Hartford, CT 06117}
\email{kacarroll@hartford.edu}
\author{Karl Petersen}
\address{Department of Mathematics,
CB 3250 Phillips Hall,
University of North Carolina,
Chapel Hill, NC 27599 USA}
\email{petersen@math.unc.edu}

%\subjclass[2000]{Primary }
%    The 2010 edition of the Mathematics Subject Classification is
%    now available.  If you are citing a classification from the
%    new scheme, use the following input coding instead.
\subjclass[2010]{Primary (37B10)}

\date{\today}

\begin{abstract}
{Markov diagrams provide a way to understand the structures of topological dynamical systems. We examine the construction of such diagrams for subshifts, including some which do not have any nontrivial Markovian part, in particular Sturmian systems and some substitution systems. }
\end{abstract}

\maketitle


\section{Introduction}
F. Hofbauer \cite{Hofbauer} and J. Buzzi \cite{Buzzi} defined Markov diagrams in order to study the structures {and invariant measures} of dynamical systems, especially those with a Markovian aspect, for example piecewise monotonic interval maps and other possibly nonunifomly expanding maps. Here we examine further the construction of these diagrams for subshifts, including some that are minimal and have zero entropy. Such subshifts may be considered to be highly non-Markovian, since they have some long-range order, indeed infinite memory. {We hope that Markov diagrams will be useful also for understanding and classifying such systems, for example besides helping to identify measures of maximal entropy as in \cite{Hofbauer, Hofbauer-Raith, Buzzi-Hubert} also to determine complexity functions, estimate return times to cylinders, and so on.}

In Sections \ref{HB} and \ref{Sturmian} we provide a construction of Hofbauer-Buzzi Markov diagrams for Sturmian systems.
In particular, in Theorem \ref{sturmthm} we show that the Hofbauer-Buzzi Markov diagram of a Sturmian system can be constructed solely from its left special sequence.
 In Section \ref{Gen} we discuss properties of Hofbauer-Buzzi Markov diagrams that hold for any subshift.
We show that given a one-sided subshift $X^+$ there is a correspondence between those paths on the Hofbauer-Buzzi Markov diagram of $X^+$ that start with a vertex of length one and points in $X^+$ (Theorem \ref{pathprop}). Corollary \ref{complexity} relates the number of such paths to the complexity function of $X^+$. We prove that the eventually Markov part of the natural extension of any one-sided subshift is empty provided that the natural extension is an infinite minimal subshift (Proposition \ref{evMark}).
In Section \ref{Morse} we construct the Hofbauer-Buzzi Markov diagram for the Morse minimal subshift by showing that the vertices are precisely those blocks in the language of the subshift that are of the form 0 or 1 followed by a block that can be extended to the left in two ways.

\section{Background}
 We recall some of the basic terminology and notation from topological and symbolic dynamics; for more details, see for example \cite{LM} or \cite{KEP}.
 A \emph{topological dynamical system} is a pair $(X,T)$, where $X$ is a compact Hausdorff space (usually metric) and $T: X \to X$ is a continuous mapping.
We focus on topological dynamical systems which are \emph{shift dynamical systems}. Let $\ca$ be a finite set, called an \emph{alphabet}, whose elements are called {\em symbols}. For us often $\ca=\{0, 1, \dots, n-1\}$, in fact often $\ca =\{0,1\}$. A \emph{sequence} is a one-sided infinite string of symbols (a function $\Bn \to \ca$) and a \emph{bisequence} is an infinite string of symbols that extends in two directions (a function $\Bz \to \ca$).
We will use the word ``sequence" to apply also to bisequences, depending on the context to clarify the meaning.
The \emph{full $n$-shift} is $\Sigma_n=\{0,1,\dots,n-1\}^\Bz$, the collection of all bisequences of symbols from $\ca=\{0,1,\dots,n-1\}$.
The \emph{one-sided full $n$-shift} is $\Sigma^+_n=\{0,1,\dots,n-1\}^\Bn$.
We also define the \emph{shift transformation} $\sigma : \Sigma(\ca) \to \Sigma(\ca)$ and $\Sigma^+(\ca) \to \Sigma^+(\ca)$ by
$(\sigma x)_i = x_{i+1} \hspace{3mm} \text{for all } i.$ The pair $(\Sigma_n, \sigma)$ is called the $n$-\emph{shift dynamical system}.
We give $\ca$ the discrete topology and  $\Sigma(\ca)$ and $\Sigma^+(\ca)$ the product topology. The topologies on $\Sigma(\ca)$ and $\Sigma^+(\ca)$ are compatible with the metric $d(x,y) = 1/2^n$, where $n = \inf \{|k|\hspace{2mm} |\hspace{2mm} x_k \neq y_k \}$.
A \emph{subshift} is a pair $(X, \sigma)$ (or $(X^+, \sigma)$), where $X \subset \Sigma_n$ (or $X^+ \subset \Sigma_n^+$) is a nonempty, closed, shift-invariant set.
A finite string of letters from $\ca$ is called a \emph{block} and the length of a block $B$ is denoted $|B|$. Furthermore, a block of length $n$ is an $n$-block. Given a subshift $(X, \sigma)$ of a full shift, $\cl_n(X)$ denotes the set of all $n$-blocks that occur in points in $X$. The \emph{language} of $X$ is the collection
$\cl(X) = \bigcup_{n=0}^{\infty}\cl_n(X).$
A {\em shift of finite type} is a subshift consisting of all sequences none of whose subblocks are in some finite collection of forbidden blocks of finite length.
A topological dynamical system is {\em minimal} if every orbit is dense. The orbit closure of a sequence is minimal if and only if the sequence is {\em syndetically recurrent}: every block that appears in the sequence appears with bounded gap.
The \emph{complexity function} of a sequence $u$, denoted $p_u$, maps each natural number $n$ to the number of blocks of length $n$ that appear in $u$.
If $X$ is a subshift, then $p_X(n)$ is the number of blocks of length $n$ that appear in $\cl(X).$

The construction of Hofbauer-Buzzi Markov diagrams involves the use of \emph{follower sets}.
There are several ways to define follower sets.
The ({\em block to block) follower set} of a block $w \in \cl(X)$ is $F_X(w) = \{v \in \cl(X) | wv \in \cl(X) \}.$
Alternatively, define the \emph{future} $F_X$ of a left-infinite sequence $\lambda$ in $X$ to be the collection of all right-infinite sequences $\rho$ such that $\lambda \rho \in X$. This is a \emph{ray to ray} follower set. It is also possible to define \emph{block to ray} or \emph{ray to block} follower sets. The definition of follower set (\ref{fol}) used in constructing Hofbauer-Buzzi diagrams is slightly different from both of these.

Follower sets have been particularly useful in examining sofic systems.
 A \emph{sofic shift} is a shift space that is a factor of a shift of finite type \cite{Weiss}.
Alternatively, a sofic shift consists of all sequences that are labels of infinite walks on a finite graph with labeled edges (see \cite{LM}).
Fischer \cite{Fischer} and Krieger \cite{Krieger} used follower sets to construct covers for sofic shifts. A \emph{presentation} of a sofic shift $X$ is a labeled graph $\cg$ for which $X_{\cg}= X$. A presentation is \emph{right-resolving} if for each vertex $I$ of $G$ the edges starting at $I$ carry different labels. A \emph{minimal right-resolving presentation} of a sofic shift $X$ is a right-resolving presentation of $X$ having the fewest vertices among all right-resolving presentations of $X$. Fischer proved that any two minimal right-resolving presentations are isomorphic as labeled graphs; the minimal right-resolving presentation of a sofic shift $X$ is called the \emph{Fischer cover} \cite{Fischer, LM}.

Given an irreducible (topologically transitive) sofic shift $X$ over a finite alphabet $\ca$, the Fischer cover can be constructed using the follower sets defined above. Let $\cc_X$ be the collection of all (block to block) follower sets in $X$. We write $\cc_X = \{ F_X(w)| w \in \cl(X) \}$. Now construct a labeled graph $\cg = (G, L)$ as follows. The vertices of $G$ are the elements in $\cc_X$. Let $c = F_X(w)$ be an element in $\cc_X$ and $a \in \ca$. If $wa \in \cl(X)$, let $c' = F_X(wa) \in \cc_X$ and draw an edge labeled $a$ from $c$ to $c'$. If  $wa \notin \cl(X)$, do nothing. Continuing this process for all elements in $\cc_X$ yields a labeled graph $\cg$ called the \emph{follower set graph}. The Fischer cover of $X$ is the labeled subgraph of the follower set graph formed by using only the follower sets of intrinsically synchronizing blocks. Here a block $w$ in $\cl(X)$ is \emph{intrinsically synchronizing} if whenever $uw, wv \in \cl(X)$ then $uwv \in \cl(X)$ \cite{LM}.

The Krieger cover is constructed using the \emph{futures}, as defined above, of left-infinite sequences in $X$. We define the \emph{future cover} as follows. Let $\cg$ be the labeled graph whose vertices are the futures of left-infinite sequences. For $a$ in $\ca$, if $\lambda$ and $\lambda a$ are left-infinite sequences in $X$, then there is an edge labeled $a$ from $F_X(\lambda)$ to $F_X(\lambda a)$. The graph $\cg$ is the \emph{future} or Krieger cover of the subshift $X$ \cite{Krieger, LM}.
The Krieger cover can be constructed for any subshift $X$, but it usually leads to non-irreducible and often uncountable graphs. Nevertheless, the Krieger cover is canonically associated to the subshift $X$. This is proved for the sofic case in \cite{Krieger2} and in general in \cite{Fiebig}.

\section{Hofbauer-Buzzi Markov diagrams} \label{HB}
Franz Hofbauer \cite{Hofbauer} constructed Markov diagrams to determine measures of maximal entropy for piecewise monotonically increasing functions on the interval. In 1997, Buzzi extended the construction of the Hofbauer Markov diagram to arbitrary smooth interval maps \cite{Buzzi1}, and in 2010 to any subshift \cite{Buzzi}. The Hofbauer-Buzzi Markov diagram is a slight variation of Hofbauer's original Markov diagram. We will refer to such diagrams as {\em HB diagrams}. In order to describe the construction, we introduce the following definitions from \cite{Buzzi}.

Let $\mathcal{A}$ be a finite alphabet and $X^+ \subset \mathcal{A}^\Bn$ a one-sided subshift. Furthermore, let $\natext \subset \mathcal{A}^\Bz$ be its natural extension
$$\natext = \{ x \in \mathcal{A}^{\Bz} | \text{ for all } p \in \Bz \hspace{2mm} x_px_{p+1}... \in X^+\},$$
with the action of the shift $\sigma$ which maps $a_n \to a_{n+1}$ for all $n$ in $\Bz$.

\begin{definition}
Let $\pi_{X^+}$ denote the continuous shift commuting projection from $\natext$ to $X^+$ defined by
\begin{center}
$\pi_{X^+}(x) = x_0x_1x_2...$,
\end{center}
where $x=...x_{-1}.x_0x_1x_2...$.
\end{definition}

\begin{definition} \label{fol} The \emph{follower set} of a block $w=a_{-n}a_{-n+1}...a_0$ in $\cl(\natext)$ is
$$\text{fol}(a_{-n}a_{-n+1}...a_0) = \{b_0b_1... \in X^+ | \text{ there exists }  b \in \natext \text{ with } b_{-n}...b_{0} = a_{-n}...a_0 \}.$$
\end{definition}

\begin{definition}
A \emph{significant block} of $\natext$ is $a_{-n}a_{-n+1}...a_0$ such that
$$\text{fol}(a_{-n}a_{-n+1}...a_0) \subsetneq \text{fol}(a_{-n+1}a_{-n+2}...a_0).$$
\end{definition}

\begin{definition} The \emph{significant form} of $a_{-n}a_{-n+1}...a_0$ in $\natext$ is
$$\text{sig}(a_{-n}...a_0) = a_{-k}...a_0,$$
where $k \leq n$ is maximal such that $a_{-k}...a_0$ is significant.
\end{definition}

It is apparent that these definitions are tailored for one-sided subshifts. However, we can easily extend such definitions to an arbitrary two-sided subshift $X \subset \Sigma_n$ by letting $X^+$ denote the set of right rays that appear in points in $X$. Then $X$ is equal to the natural extension $\natext$ of $X^+$.

We define the \emph{HB diagram} $\mathcal{D}$ of a one or two-sided subshift $X$ with natural extension $\natext$ to be the oriented graph whose vertices are the significant blocks of $\natext$ and whose arrows are defined by
\begin{center}
$a_{-n}...a_0 \to b_{-m}...b_0$
\end{center}
if and only if  $a_{-n}...a_0b_0 \in \cl (\natext)$ and
\begin{center}
$b_{-m}...b_0 = \text{sig}(a_{-n}...a_0b_0).$
\end{center}
\noindent (In Hofbauer's construction of Markov diagrams the vertices are the follower sets, not the significant blocks \cite{Hofbauer, Buzzi}.)

Let $\cd$ be the HB diagram of any one or two-sided subshift $X$. The following definitions from \cite{Buzzi} relate $\cd$ to $\natext$.
\begin{definition}
Given an HB diagram $\cd$ of a subshift $X$ with vertex set $V_\cd$ (which may be infinite), the corresponding \emph{Markov shift} is the set of all bi-infinite paths that occur on $\cd$,
\begin{center}
$\displaystyle X(\cd) = \{\alpha \in V_\cd^{\Bz} \hspace{1mm}| \text{ for all } p \in \Bz \hspace{2mm} \alpha_p \to \alpha_{p+1} \text{ on } \cd \} \subset (V_\cd)^{\Bz}$,
\end{center}
 together with the shift map $\sigma$.
\end{definition}

 Note that the alphabet $V_\cd$ may be infinite, and the HB diagram of an arbitrary subshift may not have paths that continue infinitely in two directions.

We relate $X(\cd)$ to $\natext$ as follows.
\begin{definition}
Let $\hat{\pi}$ denote the natural continuous projection defined by
$$\hat{\pi}: \alpha \in X(\cd) \mapsto a \in \natext$$ with $a_n$ the last symbol of the block $\alpha_n$ for all $n \in \Bz$.
\end{definition}

\begin{definition}\label{def:one-sidedpaths}
Let $X(\cd)_v^+$ denote the space of one-sided infinite paths starting at vertex $v$ on $\cd$, and let $X(\cd)^+$ denote the space of one-sided infinite paths starting at a vertex $v$ of length 1 on $\cd$. %$ X(\cd)^+ = \{\alpha \in X(\cd)_v^+ | |v| =1 \}.$
\end{definition}

\begin{definition}
Let $\hat{\pi}^+$ denote the projection defined by
$$\hat{\pi}^+: \alpha \in X(\cd)^+ \mapsto a \in X^+$$ with $a_n$ the last symbol of the block $\alpha_n$ for all $n \in \Bn.$
\end{definition}

In case we want to project a finite path $\alpha_0 \to \alpha_1 \to \cdots \to \alpha_n$ on $\cd$ to a block in $\cl(X^+)$, we write $\hat{\pi}(\alpha_0...\alpha_n) = a_0...a_n,$ where $a_i$ is the last letter of $\alpha_i$.

We state a few preliminary results that apply to any subshift.

\begin{lemma}
If $a_1...a_n$ is a signficant block of a subshift $X$, then $a_1...a_{n-1}$ is also a significant block of $X$.
\label{consecsig}
\end{lemma}

\begin{proof}
If $a_1...a_n$ is significant, then there exists a ray $c_0c_1c_2... \in \fol(a_2...a_n)$ such that $c_0c_1c_2... \notin \fol(a_2...a_n).$
Thus, there does not exist $c \in \natext$ with $c_{-n+1}...c_0 = a_1...a_n$.

Consider $a_1...a_{n-1}.$ Certainly, $c_{-1}c_0c_1... \in \fol(a_2...a_{n-1}).$ Suppose on the contrary that $c_{-1}c_0c_1... \in \fol(a_1...a_{n-1})$. Then there exists a $b \in \natext$ with $b_{-n+2}...b_0 = a_1...a_{n-1}$ and $b_0b_1b_2... = c_{-1}c_0c_1...$. However, $b_{-n+2}...b_0b_1 = a_1...a_n$ and $b_1b_2... = c_0c_1...$. Relabeling, this implies that $c_0c_1... \in \fol(a_1...a_n)$. This is a contradiction.

Thus $c_{-1}c_0c_1... \notin \fol(a_1...a_{n-1})$ and $a_1...a_{n-1}$ is a significant block of $X$.
\end{proof}

The following Proposition is an immediate consequence of Lemma \ref{consecsig}.

\begin{prop}
If there are infinitely many significant blocks of a subshift $X$, then for all $n \in \Bn$ there exists a significant block of $X$ of length $n$.
\end{prop}

\section{Markov diagrams for one-sided Sturmian systems}\label{Sturmian}
\subsection{Basic properties of Sturmian sequences}
We recall the definition and basic properties of Sturmian sequences; see \cite[Ch. 6]{Fogg} for details.
 A one or two-sided sequence $u$ with values in a finite alphabet is called \emph{Sturmian} if it has complexity function (defined above) $p_u(n) = n+1$ for all $n$.
If $u$ is Sturmian, then $p_u(1) = 2$. This implies that Sturmian sequences are over a two-letter alphabet, so we fix the alphabet $\ca = \{0,1\}$.
Given a one-sided Sturmian sequence $u$, we let $X_u^+$ be the closure of $\{\sigma^nu | n \in \Bn\}$. Then $(X_u^+, \sigma)$ is the {\em Sturmian system} associated with $u$.

\begin{example}
\label{fibex}
The Fibonacci substitution is defined by: %The Fibonacci sequence, $f$, is the fixed point of the Fibonacci substitution
  	\begin{align*}
  	\phi: 0 & \mapsto 01
  	\\1 & \mapsto 0.
  	\end{align*}
The fixed point of the Fibonacci substitution, $f = 0100101001001010010100100101...$, is a Sturmian sequence, and $(X_f^+, \sigma)$ is the Sturmian system associated with $f$ (see \cite{Lothaire}).
\end{example}

An infinite sequence $u$ is \emph{periodic} (respectively \emph{eventually periodic}) if there exists a positive integer $M$ such that for every $n$, $u_n = u_{n+M}$ (respectively there exists $m \in \Bn$ such that for all $|n| \geq m$, $u_n = u_{n+M}$).
A set $S$ of blocks is \emph{balanced} if for any pair of blocks $u$, $v$ of the same length in $S$, $||u|_1 - |v|_1| \leq 1$, where $|u|_1$ is the number of occurrences of 1 in $u$ and $|v|_1$ is the number of occurrences of 1 in $v$.
It follows that if a sequence $u$ is balanced and not eventually periodic then it is Sturmian. This is a result of the fact that if $u$ is aperiodic, then $p_u(n) \geq n+1$ for all $n$, and if $u$ is balanced then $p_u(n) \leq n+1$ for all $n$. In fact, it can be proved that a sequence $u$ is balanced and aperiodic if and only if it is Sturmian \cite{Lothaire}. Furthermore, any shift of a Sturmian sequence is also Sturmian.

Sturmian sequences also have a natural association to lines with irrational slope. To see this, we introduce the following definitions.
 Let $\alpha$ and $\beta$ be real numbers with $0 \leq \alpha, \beta \leq 1$. We define two infinite sequences $x_{\alpha,\beta}$ and ${x'}_{\alpha,\beta}$ by
\begin{eqnarray*}
  (x_{\alpha,\beta})_n&=& \lfloor \alpha(n+1) + \beta \rfloor - \lfloor \alpha n + \beta \rfloor    \nonumber \\
  ({x'}_{\alpha,\beta})_n&=& \lceil \alpha(n+1) + \beta \rceil - \lceil \alpha n + \beta \rceil    \nonumber \\
\end{eqnarray*}
\noindent for all $n\geq 0$.
The sequence $x_{\alpha,\beta}$ is the \emph{lower mechanical sequence} and ${x'}_{\alpha,\beta}$ is the \emph{upper mechanical sequence} with slope $\alpha$ and intercept $\beta$.
The use of the words slope and intercept in the above definitions stems from the following graphical interpretation.
 The points with integer coordinates that sit just below the line $y= \alpha x + \beta$ are $F_n = (n, \lfloor \alpha n + \beta \rfloor).$ The straight line segment connecting two consecutive points $F_n$ and $F_{n+1}$  is horizontal if  $x_{\alpha,\beta} = 0$ and diagonal if  $x_{\alpha,\beta} = 1$. The lower mechanical sequence is a coding of the line $y= \alpha x + \beta$ by assigning to each line segment connecting $F_n$ and $F_{n+1}$ a 0 if the segment is horizontal and a 1 if the segment is diagonal. Similarly, the points with integer coordinates that sit just above this line are $F'_n = (n, \lceil \alpha n + \beta \rceil)$. Again, we can code the line  $y= \alpha x + \beta$ by assigning to each line segment connecting $F'_n$ and $F'_{n+1}$ a 0 if the segment is horizontal and a 1 if the segment is diagonal. This coding yields the upper mechanical sequence \cite{Lothaire}.

A mechanical sequence is $rational$ if the line $y= \alpha x + \beta$ has rational slope and  $irrational$ if $y= \alpha x + \beta$ has irrational slope. In \cite{Lothaire} it is proved that a sequence $u$ is Sturmian if and only if $u$ is irrational mechanical. In the following example we construct a lower mechanical sequence with irrational slope, thus producing a Sturmian sequence.

\begin{example}
Let $\alpha = 1/\tau^2$, where $\tau = (1+ \sqrt{5})/2$ is the golden mean, and $\beta = 0$. The lower mechanical sequence $x_{\alpha,\beta}$ is constructed as follows:

\begin{align*}
 (x_{\alpha, \beta})_0 = & \lfloor 1/\tau^2 \rfloor = 0  \\
 (x_{\alpha, \beta})_1= & \lfloor 2/\tau^2  \rfloor - \lfloor 1/\tau^2  \rfloor = 0 \\
 (x_{\alpha, \beta})_2= & \lfloor 3/\tau^2  \rfloor - \lfloor 2/\tau^2  \rfloor = 1 \\
 (x_{\alpha, \beta})_3= & \lfloor 4/\tau^2  \rfloor - \lfloor 3/\tau^2  \rfloor =  0\\
 (x_{\alpha, \beta})_4= & \lfloor 5/\tau^2  \rfloor - \lfloor 4/\tau^2  \rfloor =  0\\
 (x_{\alpha, \beta})_5= & \lfloor 6/\tau^2  \rfloor - \lfloor 5/\tau^2  \rfloor =  1 \\
 \vdots
\end{align*}
Further calculation shows that $x_{\alpha,\beta} = 0010010100... = 0f$, and ${x'}_{\alpha,\beta} = 1010010100...= 1f,$ hence the fixed point $f$ is a shift of the lower and upper mechanical sequences with slope $1/\tau^2$ and intercept $0$.
\end{example}

We now consider the language of a Sturmian sequence $u$. It is easy to show that while Sturmian sequences are aperiodic, they are syndetically recurrent \cite{Fogg}. As a result, any block in $\cl_n(u)$ appears past the initial position and can thus be extended on the left.  Since there are $n+1$ blocks of length $n$, it must be that exactly one of them can be extended to the left in two ways.
In a Sturmian sequence $u$, the unique block of length $n$ that can be extended to the left in two different ways is called a \emph{left special block}, and is denoted $L_n(u)$. The sequence $l(u)$ which has the $L_n(u)$'s as prefixes is called the \emph{left special sequence} or \emph{characteristic word} of $X_u^+$ \cite{Fogg, Lothaire}.
 Similarly,
in a Sturmian sequence $u$, the unique block of length $n$ that can be extended to the right in two different ways is called a \emph{right special block}, and is denoted $R_n(u)$. The block $R_n(u)$ is precisely the reverse of $L_n(u)$ \cite{Fogg}.

\subsection{The left special sequence}
Since every Sturmian sequence $u$ is irrational mechanical, there is a line with irrational slope $\alpha$ associated to $u$. This $\alpha$ can be used to determine the left special sequence of $X_u^+$.

Let $(d_1, d_2, ..., d_n,...)$ be a sequence of integers with $d_1 \geq 0 $ and $d_n > 0 $ for $n>1$. We associate a sequence $(s_n)_{n\geq -1}$ of blocks to this sequence by
\begin{center}
$s_{-1} = 1$, \hspace{3mm} $s_0 = 0$, \hspace{3mm} $s_n = s^{d_n}_{n-1}s_{n-2}. $
\end{center}
The sequence $(s_n)_{n\geq -1}$ is a \emph{standard sequence}, and  $(d_1, d_2, ..., d_n,...)$  is its \emph{directive sequence}. We can then determine the left special sequence of  $X_u^+$ with the following proposition stated in \cite{Lothaire}.
\begin{prop}
Let $\alpha = [0, 1+ d_1, d_2, ....]$ be the continued fraction expansion of an irrational $\alpha$ with $0< \alpha < 1$, and let $(s_n)$ be the standard sequence associated to $(d_1, d_2,...)$. Then every $s_n$, $n \geq 1$, is a prefix of $l$ and
$$l = \lim_{n\to \infty} s_n.$$
\label{lss}
\end{prop}
This is illustrated in the following two examples.

\begin{example}
\label{fibex2}
Let  $\alpha = 1/\tau^2$, where $\tau = (1+ \sqrt{5})/2$ is the golden mean. The continued fraction expansion of $1/\tau^2$ is $[ 0,2,\overline{1} ]$. By the above proposition $d_1 = 1, d_2= 1, \\ d_3 = 1, d_4 = 1, ... $. The standard sequence associated to  $(d_1, d_2,...)$ is constructed as follows:

\begin{center}
\begin{align*}
s_1 =  & s_0^{d_1}s_{-1} =  01
\\ s_2 = & s_1^{d_2}s_0 =   010
\\ s_3 = & s_2^{d_3}s_1 =   01001
\\ s_4 = & s_3^{d_4}s_2 =   01001010
\\ \vdots
\end{align*}

\end{center}
Continuing this process, the left special sequence of $X_u^+$, where $u$ is a coding of a line with slope $1/\tau^2$, is
$$l = 010010100100101001 ... = f.$$ It follows that the left special sequence of $X_f^+$ is $f$.
\end{example}

\begin{example}
Let $\alpha = \pi/4$. The continued fraction expansion of $\pi/4$ is \[ [0, 1,3,1,1, 1, 15, 2, 72, ...] .\] By Proposition \ref{lss} $d_1 = 0, d_2 = 3, d_3 = 1, d_4 = 1, ...$. Then,

\begin{center}
\begin{align*}
s_1 =  & s_0^{d_1}s_{-1} =  1
\\ s_2 = & s_1^{d_2}s_0 =   111 0
\\ s_3 = & s_2^{d_3}s_1 =   1110 1
\\ s_4 = & s_3^{d_4}s_2 =   11101 1110
\\ \vdots
\end{align*}

\end{center}
Continuing this process, the left special sequence of $X_u^+$, where $u$ is a coding of a line with slope $\pi/4$, is
$$l = 11101111011101111011110....$$
\end{example}

\subsection{Significant blocks of a one-sided Sturmian system}
In order to construct the HB diagram of a Sturmian system, it is necessary to identify the significant blocks of the system. We first note that if  $X$ is any subshift of $\Sigma_d$, then $0, 1,..., d-1$, and $d$ are significant blocks provided $0, 1,..., d-1$, and $d$ are all in $\cl(\natext)$. Hence, $0$ and $1$ are significant blocks of any Sturmian system. Let $(X_u^+, \sigma)$ be a Sturmian system with $l = l_1l_2l_3...$ the left special sequence of $u$. In the next two propositions we prove that given $n \geq 1$, there are exactly two significant blocks of $\tilde{X}_u$ with length $n$.

The first proposition applies to any subshift of $\Sigma_2.$
\begin{prop}
Let $X \subset \Sigma_2$ be a subshift. Suppose $a_{-n+1}a_{-n+2}...a_{-1}a_0$ is a block of length $n$ in $\cl(\natext)$. If $a_{-n+1}a_{-n+2}...a_{-1}a_0$ is significant, then $0a_{-n+2}...a_{-1}a_0$ and $1a_{-n+2}...a_{-1}a_0$ are in $\cl(\natext)$.
\label{leftextend}
\end{prop}

\begin{proof}
Assume that the block $a_{-n+1}a_{-n+2}...a_{-2}a_{-1}a_{0}$ is significant.  Suppose on the contrary that  $0a_{-n+2}...a_{-2}a_{-1}a_{0}$ and $1a_{-n+2}...a_{-2}a_{-1}a_{0}$ are not both in $\cl(\natext)$. Without loss of generality, suppose $1a_{-n+2}...a_{-2}a_{-1}a_{0} \notin \cl(X)$. So,
\begin{center} $a_{-n+1}a_{-n+2}...a_{-2}a_{-1}a_{0} = 0a_{-n+2}...a_{-2}a_{-1}a_{0}$. \end{center}

Let $b_0b_1...  \in \text{fol} (a_{-n+2}a_{-n+3}...a_0)$. Then there exists $b$, a two-sided sequence in $\natext$, such that $ b_{-n+2}...b_{0} = a_{-n+2}...a_0$. However, $1a_{-n+2}...a_{-2}a_{-1}a_{0} \notin \cl(X)$ implies that $b_{-n+1} =0$. Then $b$ is such that $ b_{-n+1}...b_{0} = a_{-n+1}...a_0$ and so $b_{0}b_{1}... \in \text{fol} (a_{-n+1}a_{-n+3}...a_0)$. Thus $\text{fol} (a_{-n+2}a_{-n+3}...a_0) \subseteq \text{fol}(a_{-n+1}a_{-n+2}...a_0)$. Since it is always the case that $\text{fol} (a_{-n+1}a_{-n+2}...a_0)  \subseteq \text{fol}(a_{-n+2}a_{-n+3}...a_0)$, this implies
\begin{center}$\text{fol} (a_{-n+1}a_{-n+2}...a_0)  = \text{fol}(a_{-n+2}a_{-n+3}...a_0)$, \end{center}
contradicting $a_{-n+1}a_{-n+2}...a_{-2}a_{-1}a_{0}$ being a significant block.
\end{proof}

Before we state the second proposition, recall that since $u$ is recurrent, every block is extendable to the left. Hence, $\cl(u) = \cl(X_u^+) = \cl(\natext_u).$

\begin{prop} Let $u$ be a Sturmian sequence, with $(X_u^+, \sigma)$ its associated Sturmian system. Then for each $n \geq 2$ there are exactly two significant blocks of $\tilde{X}_u$ of length $n$. The two significant blocks of length $n$ are $0L_{n-1} = 0l_1...l_{n-1}$ and $1L_{n-1}=1l_1...l_{n-1}$.
\label{Prop1}
\end{prop}
\begin{proof}
By Proposition \ref{leftextend} we know that if the block $a_{-n+1}a_{-n+3}...a_{-2}a_{-1}a_{0}$ is significant then $0a_{-n+2}...a_{-2}a_{-1}a_{0}$ and $1a_{-n+2}...a_{-2}a_{-1}a_{0}$ are in $\cl(\natext_u)$. Thus the only possible significant blocks of length $n$ are those blocks  $a_{-n+1}...a_{-2}a_{-1}a_{0}$ such that  $a_{-n+2}...a_{-2}a_{-1}a_{0}$ can be extended to the left in two ways. That is, the possible significant blocks are $a_{-n+1}a_{-n+2}...a_{-2}a_{-1}a_{0}$ with $$a_{-n+2}...a_{-2}a_{-1}a_{0} = L_{n-1}.$$

Let $L_{n-1} = a_{-n+2}...a_{0}$ and $\nu \in \mathcal{L}(u)$. Since there is exactly one right special block of each length $n \in \Bn$, $0L_{n-1}\nu$ and $1L_{n-1}\nu$ cannot both be right special blocks. We first prove that $1L_{n-1}$ is significant by considering the following two cases.

\emph{Case 1:} There exists $\nu \in \cl(u)$ such that $0L_{n-1}\nu$ is right special. Then $1L_{n-1}\nu$ is not right special. This implies that $1L_{n-1}\nu1$ is not in $\mathcal{L}(u)$, since $$||0L_{n-1}\nu0|_1 - |1L_{n-1}\nu1|_1| = 2,$$ which is not permitted as $u$ is balanced.  Thus there exists a ray $a_{0}\nu1... $ in  $\text{fol}(L_{n-1})$ that is not in $\text{fol}(1L_{n-1})$. Hence, $1L_{n-1}$ is significant.

\emph{Case 2:} There does not exist $\nu \in \cl(u)$ such that $0L_{n-1}\nu$ is right special. This implies that there exists exactly one ray $b_1b_2...$ that can follow $0L_{n-1}$. We claim that because $u$ is Sturmian such a case cannot occur.

Let $u = c_0c_1c_2c_3...$. Since $0L_{n-1} \in \cl(u)$, we know that $0L_{n-1}$ appears in $u$ infinitely many times. Suppose $0L_{n-1}$ appears for the first time starting at position $c_{m+1}$. Letting $0L_{n-1} = a_{-n+2}...a_{0}$, we have
$$u =c_0c_1c_2c_3...c_m0L_{n-1}b_1b_2... = c_0c_1c_2c_3...c_m0a_{-n+2}...a_{0}b_1b_2... $$

Furthermore, there exists $r \in \mathbb{N}$ such that $0L_{n-1}$ appears again starting at $b_{r+1}.$ As $0L_{n-1}$ can be followed only by $b_1b_2...$, this implies that
$$u = c_0c_1c_2c_3...c_m0a_{-n+2}...a_{0}b_1b_2...b_r0a_{-n+2}...a_{0}b_1b_2...b_r0a_{-n+2}...a_{0}b_1b_2...b_r...$$

Letting $B = 0a_{-n+2}...a_{0}b_1b_2...b_r$, we have that
$$u = c_0c_1c_2c_3...c_mBBBB... .$$
Thus, $u$ is eventually periodic. This, however, is a contradiction as Sturmian sequences are not eventually periodic. Hence $1L_{n-1}$ is a significant block of $\natext_u$.

By the same argument, it can be shown that $0L_{n-1}$ is also a significant block of $\natext_u$.
\end{proof}

\subsection{Construction of the diagram}
Recall that the HB diagram of a one-sided Sturmian system is defined to be an oriented graph whose vertices are the significant blocks of $\tilde{X}_u$ and whose arrows are defined by $$a_{-n}...a_0 \to b_{-m}...b_0$$ if and only if  $a_{-n}...a_0b_0 \in \cl (\tilde{X}_u)$ and $$b_{-m}...b_0 = \text{sig}(a_{-n}...a_0b_0).$$ Having determined the significant blocks of $\tilde{X}_u$, it remains only to determine the arrows. This will give us a complete description of the HB diagram of an arbitrary one-sided Sturmian system.

Let $l = l_1l_2l_3...$ be the left special sequence of $X_u^+$. We first consider the arrows leaving the significant blocks of length 1.

\begin{lemma}
If $l_1 = 0$, then $0 \to 1$, $0 \to 00$ and $1 \to 10$. If $l_1 = 1$, then $1 \to 0$, $ 1 \to 11$ and $0 \to 01$.
\label{Lemma1}
\end{lemma}

\begin{proof}
Suppose $l_1 = 0$. By definition $$0 \to b_{-m}...b_0$$ if and only if  $0b_0 \in \cl (\tilde{X}_u)$ and $$b_{-m}...b_0 = \text{sig}(0b_0).$$  As $b_0 \in \{0,1\}$, we consider $\text{sig}(00)$ and $\text{sig}(01)$. Proposition \ref{Prop1} implies that the significant blocks of length two are 00 and 10, since $l_1 = 0$. Thus  $\text{sig}(00) = 00$ and $\text{sig}(01) = 1$. Hence, $0 \to 1$ and $0 \to 00$. Additionally, consider $\sig(1b_0)$. Since $0$ is the unique right special block of length one, the balance property implies that $11 \notin \cl(\tilde{X}_u)$. Thus there is exactly one arrow leaving the block $1$, $1 \to \sig(10)$, where $\sig(10)=10$ by Proposition \ref{Prop1}.

Similarly, if  $l_1 = 1$ we consider $\text{sig}(10)$ and $\text{sig}(11)$. In this case, Proposition \ref{Prop1} implies that the significant blocks of length two are 10 and 11. Thus  $\text{sig}(10) = 0$ and $\text{sig}(11) = 11$ and $1 \to 0$ and $ 1 \to 11$. Furthermore, the only arrow leaving $0$ is given by $0 \to \sig(01)$, where $\sig(01)=01$ by Proposition \ref{Prop1}.

\end{proof}

Now consider an arbitrary significant block $xl_1l_2...l_{n-1}$, where $x$ is either 0 or 1. Again,  $$xl_1l_2...l_{n-1} \to b_{-m}...b_0$$ if and only if  $xl_1l_2...l_{n-1}b_0 \in \cl (\tilde{X}_u)$ and $$b_{-m}...b_0 = \text{sig}(xl_1l_2...l_{n-1}b_0).$$

We consider what may follow $xl_1l_2...l_{n-1}$. There can be at most two arrows out of $xl_1l_2...l_{n-1}$, as $b_0 \in \{0,1\}$. It is always the case that $xl_1l_2...l_{n-1}l_n \in \cl(\tilde{X}_u)$. Letting $b_0 = l_n$, we get $$xl_1l_2...l_{n-1} \to b_{-m}...b_{-1}l_n$$ if and only if  $$b_{-m}...b_{-1}l_n = \text{sig}(xl_1l_2...l_{n-1}l_n).$$ However, $xl_1l_2...l_{n-1}l_n$ is significant; thus it must be that $b_{-m}...b_0 = xl_1l_2...l_{n-1}l_n$. Hence, we are guaranteed the arrow $$xl_1l_2...l_{n-1} \to xl_1l_2...l_{n}.$$ This is stated below.

\begin{lemma}
Let $xL_{n-1} = xl_1l_2...l_{n-1}$, $n>1$, $x \in \{0,1\}$, be a significant block of $\tilde{X}_u$. Then
  $$xl_1l_2...l_{n-1} \to xl_1l_2...l_{n}.$$
\label{Lemma2}
\end{lemma}
\noindent It follows from Lemma \ref{Lemma2} that the left special sequence of $X_u^+$ is seen in the diagram by reading off the last symbol in the paths $0l_1 \to 0l_1l_2 \to 0l_1l_2l_3 \to \cdots$ and $1l_1 \to 1l_1l_2 \to 1l_1l_2l_3 \to \cdots$.

Now suppose $xl_1l_2...l_{n-1}y \in \cl(\tilde{X}_u)$ and $y \neq l_n$. This occurs if and only if $xl_1l_2...l_{n-1}$ is a right special block. Since $R_n(u)$ is the reverse of $L_n(u) = l_1l_2...l_n$, we get the following lemma.

\begin{lemma}
Let $xl_1l_2...l_{n-1}$ be a significant block of $\tilde{X}_u$. In the HB diagram of $X_u^+$, two arrows leave $xl_1l_2...l_{n-1}$ if and only if $xl_1l_2...l_{n-1}$ is a right special block, equivalently if and only if $xl_1l_2...l_{n-1}= l_nl_{n-1}...l_2l_1$.
\label{Lemma3}
\end{lemma}

Suppose $xl_1l_2...l_{n-1}$ is a right special significant block, where $n>1$. Let $wl_1...l_{m-1}$, $1 \leq m < n$ be the \emph{previous} right special significant block of $\tilde{X}_u$. In other words, there is no right special significant block of length greater than $m$ and less than $n$. By definition $xl_1l_2...l_{n-1}= l_nl_{n-1}...l_2l_1$ and $wl_1...l_{m-1} = l_ml_{m-1}...l_2l_1$. We claim the following.

\begin{lemma}
Let $xl_1l_2...l_{n-1}$ and $wl_1...l_{m-1}$ be consecutive right special significant blocks as described and suppose $y \neq l_{n}$. Then
$$\text{\emph{sig}}(xl_1l_2...l_{n-1}y) = \text{\emph{sig}}(wl_1...l_{m-1}y).$$
\label{intsig}
\end{lemma}
\begin{proof}
Since $y \neq l_{n}$ it follows that $\sig(xl_1l_2...l_{n-1}y) \neq xl_1l_2...l_{n-1}y$. Suppose to the contrary that $$\text{sig}(xl_1l_2...l_{n-1}y)= \text{sig}(l_nl_{n-1}...l_2l_1y) = l_{m+i}...l_ml_{m-1}...l_2l_1y,$$
for some $i \geq 1$ with $m+i < n$.

Then $l_{m+i}...l_ml_{m-1}...l_2l_1y = zl_1...l_{m+i}$ for some $z \in \{0,1\}$, since $l_{m+i}...l_ml_{m-1}...l_2l_1y$ is significant. However, this implies that
$$l_{m+i}...l_ml_{m-1}...l_2l_1 = zl_1...l_{m+i-1}$$ is right special. This is a contradiction, since $l_{m+i}...l_ml_{m-1}...l_2l_1$ is a right special block of length $m+i$, with $m < m+i < n$, and $wl_1...l_{m-1} = l_ml_{m-1}...l_2l_1$ is the previous right special significant block.
Hence, $$\text{\emph{sig}}(xl_1l_2...l_{n-1}y) = \text{\emph{sig}}(wl_1...l_{m-1}y).$$
\end{proof}

We use the following lemma to determine the remaining arrows.

\begin{lemma}
Let  $xl_1l_2...l_{n-1}$ and $wl_1l_2...l_{m-1}$, $1\leq m < n$ be consecutive right special significant blocks as described and suppose $y \neq l_n$.
If $x \neq w$, then
$$x l_1l_2...l_{n-1}\to wl_1l_2...l_{m-1}l_{m}.$$
If $x=w$ then,
$$xl_1l_2...l_{n-1} \to \emph{\text{sig}}(wl_1l_2...l_{m-1}y).$$
\label{Lemma4}
\end{lemma}

\begin{proof}
Suppose $x \neq w$.

We know that $xl_1l_2...l_{n-1}= l_nl_{n-1}...l_2l_1$ and $wl_1...l_{m-1} = l_ml_{m-1}...l_2l_1$, so $x \neq w$ implies that $l_n \neq l_m$.
Additionally, $y \neq l_n$ implies $y = l_m$. By Lemma \ref{intsig},
$$\text{sig}(xl_1l_2...l_{n-1}y) = \sig(wl_1...l_{m-1}y) = \text{sig}(l_ml_{m-1}...l_2l_1l_m).$$
By Proposition \ref{Prop1} $$ \text{sig}(wl_1...l_{m-1}l_m) = wl_1...l_{m-1}l_m.$$
This gives us the arrow $$xl_1l_2...l_{n-1} \to wl_1...l_{m-1}l_m.$$

Now suppose $x = w$.
If $x = w$ then $l_n = l_m$, and thus $y \neq l_m$. Therefore
$$xl_1l_2...l_{n-1} \to  \text{sig}(wl_1...l_{m-1}y).$$
That is, there is an arrow leaving $xl_1l_2...l_{n-1}$ that points to the same significant block as one of the arrows leaving $wL_{m-1}$.
\end{proof}

We summarize the construction of the HB diagram of an arbitrary one-sided Sturmian system in the following theorem.

\begin{theorem}
Let $X_u^+$ be a one-sided Sturmian system, with $l = l_1l_2l_3...$ the left special sequence of $u$, {and $L_n=l_1 \dots l_n$ for each $n \geq 1$.} The HB diagram of $X_u^+$ is the directed graph with vertices $0, 1$, $0L_n$, and $1L_n$, $n\geq 1$, and whose arrows are defined by
\begin{enumerate}
\item $0 \to 1$, $0 \to 00$, and $1 \to 10$ if $l_1 = 0$, and $1 \to 0$, $1 \to 11$, and $0 \to 01$ if $l_1 = 1$,
\item $0L_n \to 0L_{n+1}$, $1L_n \to 1L_{n+1}$,
\item If $xL_n$ and $wL_m$, $n\geq m$, are consecutive right special blocks
   \begin{enumerate}
	\item $xL_n \to wL_{m+1}$ if $x\neq w$
	\item $xL_n \to \emph{\sig}(wL_{m}y)$, $y \neq l_{m+1}$, if $x=w.$
   \end{enumerate}
\end{enumerate}
\label{sturmthm}
\end{theorem}

We describe the construction of the HB diagrams of two Sturmian systems. Recall from Example \ref{fibex} that the Fibonacci Sturmian system is $(X_f^+, \sigma)$, where \\ $f = 0100101001001010010100...$. In Example \ref{fibex2} it is shown that the left special sequence of $X_f^+$ is $f$.
By Proposition \ref{Prop1}, the significant blocks of $\tilde{X}_f$ are $$0, 1, 00, 10, 001, 101, 0010, 1010, 00100, 10100,....$$
Furthermore, the first few right special significant blocks are $0, 10,$ $0010$, and $1010010$. Following Theorem \ref{sturmthm}, we construct a portion of the HB diagram of $X_f^+$, as depicted in Figure \ref{fig:Fib}.

\tikzstyle{int}=[draw, fill=blue!20, minimum size=2em]
\tikzstyle{init} = [pin edge={to-,thin,black}]

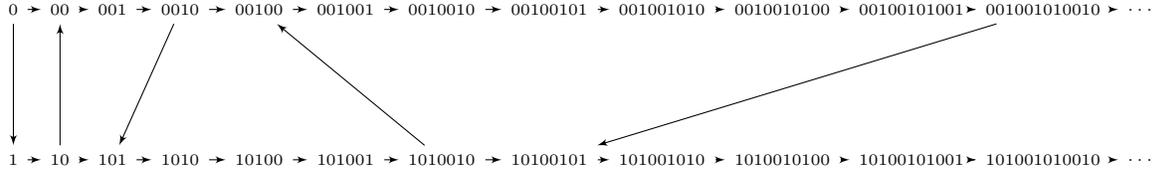
\begin{figure}[h]
\begin{center}

\begin{tikzpicture}[node distance=2.5cm,auto,>=latex']

    \node (a) {\tiny0};
    \node (b)  [right of = a, node distance=.62cm]{\tiny00};
    \node (c) [right of = b, node distance=.7cm]{\tiny001};
    \node (d)  [right of = c, node distance=.9cm]{\tiny0010};
    \node (e) [right of = d, node distance=1.05cm]{\tiny00100};
    \node (f) [right of = e, node distance=1.15cm]{\tiny001001};
    \node (g) [right of = f, node distance=1.28cm]{\tiny0010010};
    \node (h) [right of = g, node distance=1.43cm]{\tiny00100101};
    \node (i) [right of = h, node distance=1.5cm]{\tiny001001010};
    \node (j) [right of = i, node distance=1.6cm]{\tiny0010010100};
    \node (k) [right of = j, node distance=1.72cm]{\tiny00100101001};
    \node (l) [right of = k, node distance=1.75cm]{\tiny001001010010};

    \node (m) [right of = l, node distance=1.3cm]{\tiny$\dots$};

    \node  (A) [below of= a, node distance=2cm]{\tiny1};
    \node (B)  [below of = b, node distance=2cm]{\tiny10};
    \node (C) [below of = c, node distance=2cm]{\tiny101};
    \node (D)  [below of = d, node distance=2cm]{\tiny1010};
    \node (E) [below of = e, node distance=2cm]{\tiny10100};
    \node (F) [below of = f, node distance=2cm]{\tiny101001};
    \node (G) [below of = g, node distance=2cm]{\tiny1010010};
    \node (H) [below of = h, node distance=2cm]{\tiny10100101};
    \node (I) [below of = i, node distance=2cm]{\tiny101001010};
    \node (J) [below of = j, node distance=2cm]{\tiny1010010100};
    \node (K) [below of = k, node distance=2cm]{\tiny10100101001};
    \node (L) [below of = l, node distance=2cm]{\tiny101001010010};

    \node (M) [below of = m, node distance=2cm]{\tiny$\dots$};

    \path[->] (a) edge node {} (b);
       \path[->] (b) edge node {} (c);
          \path[->] (c) edge node {} (d);
             \path[->] (d) edge node {} (e);
                \path[->] (e) edge node {} (f);
                   \path[->] (f) edge node {} (g);
                      \path[->] (g) edge node {} (h);
                         \path[->] (h) edge node {} (i);
                            \path[->] (i) edge node {} (j);
                               \path[->] (j) edge node {} (k);
                                 \path[->] (k) edge node {} (l);
                                   \path[->] (l) edge node {} (m);

      \path[->] (A) edge node {} (B);
       \path[->] (B) edge node {} (C);
          \path[->] (C) edge node {} (D);
             \path[->] (D) edge node {} (E);
                \path[->] (E) edge node {} (F);
                   \path[->] (F) edge node {} (G);
                      \path[->] (G) edge node {} (H);
                         \path[->] (H) edge node {} (I);
                            \path[->] (I) edge node {} (J);
                               \path[->] (J) edge node {} (K);
                                 \path[->] (K) edge node {} (L);
                                   \path[->] (L) edge node {} (M);

     \path[->] (a) edge node {} (A);
     \path[->] (B) edge node {} (b);
     \path[->] (d) edge node {} (C);
     \path[->] (G) edge node {} (e);
     \path[->] (l) edge node {} (H);

\end{tikzpicture}

 \caption{The HB diagram of $X_f$.} \label{fig:Fib}
\end{center}
\end{figure}

Next consider the sequence $u$, where $u$ is the upper or lower mechanical sequence with slope $\alpha = \pi/4$ and intercept $\beta$, $\beta \leq 1$. Let $(X_u^+, \sigma)$ be the Sturmian system associated with the sequence $u$. Earlier we found that the left special sequence of $X_u^+$ is $$l = 11101111011101111011110....$$ Applying Proposition \ref{Prop1}, the significant blocks of $\tilde{X_u}$ are $$0,1,01,11,011,111,0111,1111,01110,11110,...,$$ and the first few right special significant blocks are $1, 11, 111$, $0111$ and $11110111$. Following Theorem \ref{sturmthm}, we begin construction of the HB diagram of $X_u^+$, as depicted in Figure \ref{fig:Pi/4}.

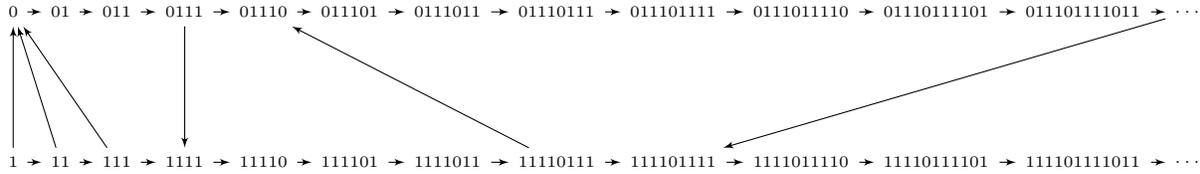
\begin{figure}[h!]
\begin{center}
%\scalebox{.6}{
\begin{tikzpicture}[node distance=2.5cm,auto,>=latex']

    \node (a) {\tiny 0};
    \node (b)  [right of = a, node distance=.63cm]{\tiny 01};
    \node (c) [right of = b, node distance=.75cm]{ \tiny 011};
    \node (d)  [right of = c, node distance=.9cm]{\tiny 0111};
    \node (e) [right of = d, node distance=1.05cm]{\tiny 01110};
    \node (f) [right of = e, node distance=1.15cm]{\tiny011101};
   \node (g) [right of = f, node distance=1.3cm]{\tiny 0111011};
  \node (h) [right of = g, node distance=1.45cm]{\tiny 01110111};
   \node (i) [right of = h, node distance=1.55cm]{\tiny 011101111};
   \node (j) [right of = i, node distance=1.7cm]{\tiny 0111011110};
    \node (k) [right of = j, node distance=1.8cm]{\tiny 01110111101};
    \node (l) [right of = k, node distance=1.95cm]{\tiny 011101111011};
    \node (m) [right of = l, node distance=1.4cm]{\tiny$\dots$};
       \node  (A) [below of= a, node distance=2cm]{ \tiny 1};
       \node (B)  [below of = b, node distance=2cm]{\tiny 11};
      \node (C) [below of = c, node distance=2cm]{\tiny 111};
   \node (D)  [below of = d, node distance=2cm]{\tiny 1111};
    \node (E) [below of = e, node distance=2cm]{\tiny11110};
   \node (F) [below of = f, node distance=2cm]{\tiny111101};
    \node (G) [below of = g, node distance=2cm]{\tiny 1111011};
   \node (H) [below of = h, node distance=2cm]{\tiny 11110111};
  \node (I) [below of = i, node distance=2cm]{\tiny 111101111};
    \node (J) [below of = j, node distance=2cm]{\tiny 1111011110};
    \node (K) [below of = k, node distance=2cm]{\tiny 11110111101};
    \node (L) [below of = l, node distance=2cm]{\tiny 111101111011};
   \node (M) [below of = m, node distance=2cm]{\tiny$\dots$};
    \path[->] (a) edge node {} (b);
       \path[->] (b) edge node {} (c);
          \path[->] (c) edge node {} (d);
             \path[->] (d) edge node {} (e);
                \path[->] (e) edge node {} (f);
                   \path[->] (f) edge node {} (g);
                      \path[->] (g) edge node {} (h);
                         \path[->] (h) edge node {} (i);
                            \path[->] (i) edge node {} (j);
                               \path[->] (j) edge node {} (k);
                                 \path[->] (k) edge node {} (l);
                                   \path[->] (l) edge node {} (m);
      \path[->] (A) edge node {} (B);
       \path[->] (B) edge node {} (C);
          \path[->] (C) edge node {} (D);
             \path[->] (D) edge node {} (E);
                \path[->] (E) edge node {} (F);
                   \path[->] (F) edge node {} (G);
                      \path[->] (G) edge node {} (H);
                         \path[->] (H) edge node {} (I);
                            \path[->] (I) edge node {} (J);
                               \path[->] (J) edge node {} (K);
                                 \path[->] (K) edge node {} (L);
                                   \path[->] (L) edge node {} (M);
     \path[->] (A) edge node {} (a);
     \path[->] (B) edge node {} (a);
     \path[->] (C) edge node {} (a);
     \path[->] (d) edge node {} (D);
     \path[->] (H) edge node {} (e);
     \path[->] (m) edge node {} (I);
\end{tikzpicture}
%}
%
 \caption{The HB diagram of $X_u^+$ where $u$ is the upper or lower mechanical sequence with slope $\alpha = \pi/4$.} \label{fig:Pi/4}
\end{center}
\end{figure}

\section{General properties of HB diagrams} \label{Gen}
We consider next what HB diagrams can tell us about their associated systems. We first consider the properties of the HB diagram that hold for any subshift.

Let $X^+$ be a one-sided subshift with natural extension $\natext$ as previously defined.
\begin{lemma}
Suppose $c_kc_{k-1}...c_1c_0$ is a block in $\cl(\natext)$. Then
\begin{center}
$ \emph{\text{sig}}(\emph{\text{sig}}(c_kc_{k-1}...c_1)c_0) = \emph{\text{sig}}(c_kc_{k-1}...c_1c_0).$
\end{center}
\label{siglem}
\end{lemma}

\begin{proof}

Let $|\text{sig}(c_kc_{k-1}...c_1)|$ denote the length of $\text{sig}(c_kc_{k-1}...c_1)$. Then
\begin{center}
 $|\text{sig}(\text{sig}(c_kc_{k-1}...c_1)c_0)| \leq |\text{sig}(c_kc_{k-1}...c_1c_0)|,$
 \end{center}
 since $|\text{sig}(c_kc_{k-1}...c_1)| \leq |c_kc_{k-1}...c_1|$.

Suppose on the contrary that
\begin{center}
 $|\text{sig}(\text{sig}(c_kc_{k-1}...c_1)c_0)| < |\text{sig}(c_kc_{k-1}...c_1c_0)|$.
 \end{center}
Furthermore, suppose $\text{sig}(c_kc_{k-1}...c_1) = c_j...c_1$, where $1 \leq j \leq k$.

 Then
 \begin{center}
 $\text{sig}(\text{sig}(c_kc_{k-1}...c_1)c_0) = \text{sig}(c_j...c_1c_0) = c_m...c_1c_0$
 \end{center}
 with $0 \leq m \leq j$, and
 \begin{center}
 $\text{sig}(c_kc_{k-1}...c_1c_0) = c_r...c_1c_0$
 \end{center}
 with $m < r \leq k$.

Since $ c_r...c_1c_0$ is significant, we know that $\text{fol}(c_r...c_1c_0) \subsetneq \text{fol}(c_{r-1}...c_1c_0)$. Thus there exists a one-sided ray $b_0b_1b_2...$ in $\text{fol}(c_{r-1}...c_1c_0)$ that is not in $\text{fol}(c_r...c_1c_0)$. Furthermore, $b_0b_1b_2...$ is such that there exists a two-sided ray $b$ in the natural extension of $X^+$, with $b_{-r+1}...b_{0} = c_{r-1}...c_0$. However, this implies that $b_{-1}b_0b_1...$ is a one-sided ray  in $\text{fol}(c_{r-1}...c_1)$ that is not in $\text{fol}(c_r...c_1)$. Hence, $c_r...c_1$ is significant.

It follows that $r \leq j$. If not $\text{sig}(c_kc_{k-1}...c_1) = c_r...c_1$. However, we have assumed $r > m$. This contradicts  $\text{sig}(c_j...c_1c_0) = c_m...c_1c_0$, since $c_r...c_1c_0$ is a longer significant block than $c_m...c_1c_0$ that is also a suffix of $c_j...c_1c_0$. Thus,
 \begin{center}
  $|\text{sig}(\text{sig}(c_kc_{k-1}...c_1)c_0)| =  |\text{sig}(c_kc_{k-1}...c_1c_0)|,$
 \end{center}
and so
  \begin{center}
  $\text{sig}(\text{sig}(c_kc_{k-1}...c_1)c_0) =  \text{sig}(c_kc_{k-1}...c_1c_0)$.
 \end{center}
\end{proof}

{
The following corollary, which is stated as an exercise in \cite{Buzzi}, follows from Lemma \ref{siglem} by induction.}
\begin{cor}
Let $\alpha_0 \to \alpha_1 \to \cdots \to \alpha_n$ be a finite path on a HB diagram $\cd$. Suppose $\alpha_0 = b_{-k}...b_0$ and let $a_i$ be the last letter of $\alpha_i$. Then for all $n \in \Bn$
\begin{center}
$\alpha_n = \emph{\text{sig}}(b_{-k}...b_0a_1a_2...a_n).$
\end{center}

\label{iteratedsig}
\end{cor}

Shifting our attention to the paths on $\cd$, we define the \emph{length} of a path to be the number of vertices in the path.
From Corollary \ref{iteratedsig} we get the following result.

\begin{theorem}
Let $X$ be a subshift with HB diagram $\cd$. If $\alpha_0 \to \alpha_1 \to \cdots \to \alpha_{n-1}$ and $\beta_0 \to \beta_1 \to \cdots \beta_{n-1}$ are two paths of length $n$ on $\cd$ with $\alpha_0$ and $\beta_0$ blocks of length 1, then $\hat{\pi}(\alpha_0\alpha_1...\alpha_{n-1}) = \hat{\pi}(\beta_0\beta_1...\beta_{n-1})$ if and only if $\alpha_i = \beta_ i $ for all $0 \leq i \leq n-1$.
\label{distinctpath}
\end{theorem}

\begin{proof}
Let $a_i$ and $b_i$ be the last letters of $\alpha_i$ and $\beta_i$ respectively.

Suppose $\hat{\pi}(\alpha_0\alpha_1...\alpha_{n-1}) = \hat{\pi}(\beta_0\beta_1...\beta_{n-1})$.  By Corollary \ref{iteratedsig},
$$\alpha_i = \text{sig}(\alpha_0a_1...a_{i}) = \text{sig}(a_0a_1...a_{i})$$
and
$$\beta_i = \text{sig}(\beta_0b_1...b_{i}) = \text{sig}(b_0b_1...b_{i}).$$
Furthermore, $a_0a_1...a_{i}= \hat{\pi}(\alpha_0\alpha_1...\alpha_i)$ and $b_0b_1...b_i = \hat{\pi}(\beta_0\beta_1...\beta_i)$. As $\hat{\pi}(\alpha_0\alpha_1...\alpha_i)$ and $\hat{\pi}(\beta_0\beta_1...\beta_i)$ are prefixes of $\hat{\pi}(\alpha_0\alpha_1...\alpha_{n-1})$ and $\hat{\pi}(\beta_0\beta_1...\beta_{n-1})$ respectively, it follows that
$$\hat{\pi}(\alpha_0\alpha_1...\alpha_i) = \hat{\pi}(\beta_0\beta_1...\beta_i).$$
Hence,
$$\alpha_i = \beta_i = \sig({\hat{\pi}(\alpha_0\alpha_1...\alpha_i)}) = \sig( \hat{\pi}(\beta_0\beta_1...\beta_i)).$$

Now suppose $\alpha_i = \beta_ i $ for all $0 \leq i \leq n-1$. Then $a_i = b_i$ for all $0 \leq i \leq n-1$. Thus $a_0...a_{n-1} = b_0...b_{n-1}.$ By definition, $\hat{\pi}(\alpha_0\alpha_1...\alpha_{n-1}) = a_0...a_{n-1}$ and $\hat{\pi}(\beta_0\beta_1...\beta_{n-1}) = b_0...b_{n-1}$. Hence
$\hat{\pi}(\alpha_0\alpha_1...\alpha_{n-1}) = \hat{\pi}(\beta_0\beta_1...\beta_{n-1})$.
\end{proof}

\noindent It follows from Theorem \ref{distinctpath} that on the HB diagram $\cd$ of a subshift $X$ each distinct path of length $n$ starting at a block of length one projects to a distinct block of length $n$ in $\cl(X)$.

Now suppose $B$ is a block in $\cl(X)$. We ask, does there exist a path on $\cd$ starting with a block of length one that projects to the block $B$? In general, this is not the case.
\newline

\begin{example}
Let $u = 1\bar{0}$ and consider the system $(X^+, \sigma)$ where $X^+$ is the orbit closure of $\{\sigma^n u | n \in \Bn \}.$ We claim that $1 \notin \cl(\natext)$. For 1 to be in $\cl(\natext)$, 1 must appear in a two-sided sequence $b$ such that $b_pb_{p+1}... \in X^+$ for all $p \in \Bz$. Suppose $$b = ... b_{-n-2}1b_{-n}...b_{-2}b_{-1}b_0b_1b_1....$$ Then $b_{-n-2}$ cannot equal 0 since $01 \notin \cl(X^+)$, and  $b_{-n-2}$ cannot equal 1 since $11 \notin \cl(X^+).$ Thus, 1 does not appear in any two-sided sequence $b$ with the property that $b_pb_{p+1}... \in X^+$ for all $p \in \Bz$. Hence, $1 \notin \cl(\natext).$ This implies that 1 is not a vertex in the HB diagram $\cd$ for $X^+$. As a result, there is no path on $\cd$ starting with a block of length one that projects to any block in $\cl(X^+)$ that begins with a 1.
\end{example}

In this example, we see that the relationship between $\cl(X)$ and $\cl(\natext)$ is closely related to the paths that appear on $\cd$. If $X$ is a two-sided subshift, obviously $\natext = X$ and $\cl(X) = \cl(\natext).$ In contrast, if $X^+$ is a one-sided subshift it is not as easy to determine whether $\cl(X^+) = \cl(\natext).$ Recall that,
$$\natext = \{ x \in \mathcal{A}^\Bz | \text{ for all } p \in \Bz, \hspace{2mm} x_px_{p+1}... \in X^+\}.$$
Hence, $\cl(\natext) \subset \cl(X^+).$  We provide a construction of $\natext$ and thus a necessary and sufficient condition on $X^+$ such that $\cl(X^+)= \cl(\natext).$

Let $a^{(n)}=a_0^{(n)}a_1^{(n)}a_2^{(n)}a_3^{(n)}...$ be points in $X^+$. We construct a sequence $(x_n(a^{(n)}))$ of two-sided sequences as follows. Let $b^{(n)} = 0^{\infty}.a^{(n)}$ and set $x_n(a^{(n)}) = \sigma^nb^{(n)}.$

\begin{prop}
Let $X^+$ and  $(x_n(a^{(n)}))$ be as described above. Then $\natext$ is the set of limit points of all $(x_n(a^{(n)}))$, $a^{(n)} \in X^+$ for all $n \geq 0$.
\label{extcons}
\end{prop}

\begin{proof}
By compactness, the sequence $(x_n(a^{(n)}) )$ has a limit point $x$. That is, there exists a subsequence $(x_{n_k}(a^{(n_k)}))$ of $(x_n(a^{(n)}) )$ that converges to $x$. We claim that any such limit point is in $\natext$. Suppose on the contrary that $x \notin \natext$. Then there exists $p \in \Bz$ such that $x_px_{p+1}x_{p+2}... \notin X^+$. This, however, is impossible since the initial blocks of any right ray in $x$ can be found as the initial blocks of a point in $X^+$ by construction. Hence, $x_px_{p+1}x_{p+2}...$ is in the closure of $X^+$ and thus is in $X^+$.

Conversely, let $b=...b_{-3}b_{-2}b_{-1}.b_0b_1b_2b_3...$ be an arbitrary bisequence in $\natext$. We show that $b$ is a limit point of a subsequence  of $(x_n(a^{(n)}))$, for some $a^{(n)} \in X^+$. Let \\
$a^{(n)} = \pi_{X^+}(\sigma^{-n}b)$, where $\pi_{X^+}$ is as defined earlier. That is
\begin{align*}
 a^{(1)} &= \pi_{X^+}(\sigma^{-1}b) = b_{-1}b_0b_1... \\
 a^{(2)} &= \pi_{X^+}(\sigma^{-2}b) = b_{-2}b_{-1}b_0... \\
 \vdots \\
 a^{(n)} &= \pi_{X^+}(\sigma^{-n}b) = b_{-n}b_{-n+1}b_{-n+2}... .
\end{align*}
It follows that,
\begin{center}
$\displaystyle b = \lim_{n \to \infty} x_n(a^{(n)})$.
\end{center}
Thus any point in $\natext$ is a limit point of a subsequence of $(x_n(a^{(n)}))$.
\end{proof}

\begin{cor}
$\cl(X^+)= \cl(\natext)$ if and only if for every block $B$ in $\cl(X^+)$ and for all $n \geq 0$ there exists $a^{(n)} \in X^+$ such that $B$ appears in $a^{(n)}$ starting at position $n.$
\label{langequiv}
\end{cor}

\begin{proof}

Suppose that $B \in \cl(X^+)$ and for all $n \geq 0$ there exists $a^{(n)} = a_0^{(n)}a_1^{(n)}a_2^{(n)}... \in X^+$ such that $B$ appears in $a^{(n)}$ starting at position $n.$ Construct the sequence $(x_n(a^{(n)})$ as defined previously. For all $n \geq 0$,
\begin{center}
 $\displaystyle x_n(a^{(n)}) = 0^{\infty}a_0^{(n)}...a_{n-1}^{(n)}.Ba_{|B|+n}^{(n)}...$.
 \end{center}
Let $x$ be any limit point of the sequence $(x_n(a^{(n)}))$. Then $x \in \natext$ by Proposition \ref{extcons} and $x= ...x_{-1}x_{0}.Bx_{|B|}...$. That is, $B$ appears in $x$ starting at position 0.  Hence $B \in \cl(\natext)$.

Now assume $\cl(X^+)= \cl(\natext)$ and $B \in \cl(X^+)$ and $n \geq 0$ are given. Since $B \in \cl(\natext)$, B appears in some $x \in \natext$. Futhermore, there exists $m \in \Bz$ such that $\sigma^mx$ has $B$ appearing in position $n$. By definition, the ray $\pi_{X^+}(\sigma^mx)$ is in $X^+$. Setting $a^{(n)}= \pi_{X^+}(\sigma^{m}x)$, we have the desired result.
\end{proof}

The following Corollary is an immediate consequence of Corollary \ref{langequiv}.
\begin{cor}
$\cl(X^+)= \cl(\natext)$ if and only if every block in $\cl(X^+)$ is left extendable.  In particular, if $X^+$ is minimal, then $\cl(X^+)= \cl(\natext)$.
\label{corleftext}
\end{cor}

We now focus our attention on subshifts $X$ such that $\cl(X)= \cl(\natext).$

\begin{theorem}
Let $X$ be a one or two-sided subshift with $\cl(X)= \cl(\natext)$. Let \\
$w = w_0w_1...w_n$ be a block in $\cl(X).$ Then there exists a unique path $\alpha_0 \to \alpha_1 \to \cdots \to \alpha_n$ in the HB diagram of $X$ with $\alpha_0 = w_0$ and $\hat{\pi}(\alpha_0\alpha_1...\alpha_n) = w$.
\label{pathprop}
\end{theorem}

\begin{proof}
Let $w = w_0w_1...w_n$ be a block in $\cl(X).$  Since $w_0...w_i$ appears in $w$ for $0 \leq i \leq n$, it follows that
$w_0...w_i \in \cl(X) = \cl(\natext)$ for all $0 \leq i \leq n$. Set $\alpha_i = \sig(w_0w_1...w_i)$. By definition $\sig(w_0w_1...w_i)$ is a significant block in $\natext$ that ends with the letter $w_i$. It follows that if $\alpha_0 \to \alpha_1 \to \cdots \to \alpha_n$ is a path on $\cd$ then $\hat{\pi}(\alpha_0\alpha_1...\alpha_n) = w_0w_1...w_n.$

It remains to show that for $0 \leq i \leq n-1$ there exist arrows from $\alpha_i \to \alpha_{i+1}$ in the HB diagram of $X.$ By definition $\alpha_i \to \alpha_{i+1}$ if and only if $\alpha_{i+1} = \sig(\alpha_iw_{i+1}).$ Since $\alpha_i = \sig(w_0w_1...w_i)$, it suffices to show that $\alpha_{i+1} = \sig(\sig(w_0w_1...w_i)w_{i+1}).$ Lemma \ref{iteratedsig} implies that $$\sig(\sig(w_0w_1...w_i)w_{i+1}) = \sig(w_0w_1...w_iw_{i+1}),$$ where $$\sig(w_0w_1...w_iw_{i+1}) = \alpha_{i+1}.$$ Thus, the desired path exists. Furthermore, this path is unique by Theorem \ref{distinctpath}.
\end{proof}

While it is not true in general, here we have shown that if $B$ is a block in $\cl(X)$ then there exists a path on $\cd$ starting with a block of length one that projects to the block $B$, provided $\cl(X) = \cl(\natext)$. It immediately follows that if $X$ is a subshift with $\cl(X) = \cl(\natext)$, then for any point $x$ in the one-sided subshift $X^+$ there exists a unique path $\alpha$ in the HB diagram of $X$ starting with a block of length one such that $\hat{\pi}^+(\alpha) = x.$ Furthermore, from Theorems \ref{distinctpath} and \ref{pathprop} we get the following corollary.

\begin{cor}
Let $X$ be a one or two-sided subshift with $\cl(X)= \cl(\natext)$ and let $p_X$ be the complexity function of $X$.
Then the number of distinct paths of length $n$ that occur on $\cd$, the HB diagram of $X$, that begin with a block of length one is equal to $p_X(n).$
\label{complexity}
\end{cor}

We provide an alternate statement and proof of Corollary \ref{complexity} that is specific to Sturmian systems.

\begin{theorem}
Let $(X_u^+, \sigma)$ be a Sturmian system. In the HB diagram of $X_u^+$, for $n\geq 1$ there are $p_u(n) =n+1$ paths of length $n$ starting from either the vertex labeled 0 or the vertex labeled 1. %Furthermore, if $\alpha_0\to \alpha_2 \to \cdots \to \alpha_{n-1}$ and $\beta_0\to \beta_2 \to \cdots \to \beta_{n-1}$ are two distinct paths of length $n$ with $\alpha_0=\beta_0 \in \{0,1\}$ then $\hat{\pi}(\alpha_0\alpha_1...\alpha_{n-1}) \neq \hat{\pi}(\beta_0\beta_1...\beta_{n-1}) $.
\label{sturmpath}
\end{theorem}

\begin{proof}

We denote the number of paths of length $n$ by $P_n$. Let $n =1$. As any path can begin with 0 or 1, $P_1 = 2$. Next let $n =2$. If 0 is right special, the paths of length two are $0 \to 00, 0 \to 1$, and $1 \to 10$. If 1 is right special, the paths of length two are $1 \to 11, 1 \to 0$, and $0 \to 01$. In either case, $P_2 =3$.

We proceed by induction. Fix $n\geq 2$ and assume $P_{n-1} = p_u(n-1)$. Then $P_{n-1} = n$. We wish to show that $P_{n} = n+1$. Consider the $n$ distinct paths of length $n-1$. Because each of these paths can be continued, there are at least $n$ paths of length $n$. Suppose there are $n+2$ paths of length $n$. From Theorem \ref{distinctpath},
each distinct path of length $n$ starting with either 0 or 1 yields a distinct block of length $n$ by reading off the last symbol of every vertex encountered. Then, that there are $n+2$ paths of length $n$ implies that there are $n+2$ distinct blocks of length $n$ in $\cl(X_u^+)$. This contradicts $p_u(n) = n+1$, so $n \leq P_n < n+2$.

To prove that $P_n = n+1$, we show that exactly one of the paths of length $n-1$ can be continued in two ways. Consider the $n$ paths of length $n-1$ with initial vertex 0 or 1. Each of these paths projects to a distinct block of length $n-1$, hence there is a path corresponding to every block in $\cl_n(u)$. It follows that exactly one of these blocks is right special. Call this block $w = l_{n-1}l_{n-2}...l_1$, where $l=l_1l_2l_3...$ is the left special sequence of $u$. Let $\alpha_0 \to \cdots \to \alpha_{n-2}$ be the path that projects to $w$. That is $w = \hat{\pi}(\alpha_0...\alpha_{n-2}) = l_{n-1}l_{n-2}...l_1.$ By Corollary \ref{iteratedsig}, $\alpha_{n-2} = \sig(l_{n-1}l_{n-2}...l_0)$. However, $\sig(l_{n-1}l_{n-2}...l_0) = l_ml_{m-1}...l_0$, $0 \leq m \leq n-1$. Thus $\alpha_{n-2}$ is a right special significant block. This implies that the path $\alpha_0 \to \cdots \to  \alpha_{n-2}$ can be continued in two ways. Thus, there are exactly $n+1$ paths of length $n$, as desired.
\end{proof}
\noindent In this proof, we not only show that there are $p_u(n)$ paths of length $n$ with initial vertex 0 or 1, but we identify the path of length $n$ that extends in two ways.

It follows from Corollary \ref{complexity} that given a subshift $X$ with the property that $\cl(X) = \cl(\natext)$, we can recover the complexity function for $X$ by counting paths in the HB diagram of $X$.
In Section \ref{Morse} we construct the HB diagram of the Morse minimal subshift (see Figure \ref{fig:MMS}).
 The complexity function of the Morse minimal subshift is  given by (see \cite[Ch. 5]{Fogg})
 $p_\omega(1) =2$, $p_\omega(2) =4$ and for $n \geq 3$ if $n=2^r+q+1$, $r \geq 0$, $0<q\leq 2^r$, then
\begin{center}
$p_\omega(n) =
\begin{cases}
6(2^{r-1})+4q& \text{if } 0 < q \leq 2^{r-1} \\
8(2^{r-1})+2q & \text{if } 2^{r-1}<q \leq 2^r.\\
\end{cases}$
\end{center}

\noindent It is apparent in examining the portion of the HB diagram of the Morse minimal subshift shown in Figure \ref{fig:MMS} that the number of paths with initial vertex 0 or 1 is equal to $p_\omega(n)$ for $n \leq 8$.

Let $X$ be any subshift and $\cd$ its HB diagram. Recall that $X(\cd)$ is the set of all bi-infinite paths that occur on $\cd$.
\begin{definition}
A sequence $a \in \natext$ is \emph{eventually Markov} at time $p \in \Bz$ if there exists $N=N(x,p)$ such that for all $n \geq N$
\begin{center}
$\fol(a_{p-n}...a_p) = \fol(a_{p-N}...a_p).$
\end{center}
The \emph{eventually Markov part} $\natext_M \subset \natext$ is the set of $a \in \natext$ which are eventually Markov at all times $p \in \Bz$.
\end{definition}
The following theorem, due to Hofbauer and Buzzi,  shows that $\hat{\pi}: X(\cd) \to \natext_M$ is an isomorphism \cite{Buzzi, Buzzi1,Hofbauer}.
\begin{theorem}
The natural projection $\hat{\pi}$ from the Hofbauer shift $X(\cd)$ to the subshift $\natext$ defined by
\begin{center}
$\hat{\pi}: \alpha \in X(\cd) \mapsto a \in \natext$
\end{center}
 with $a_n$ the last symbol of the block $\alpha_n$ for all $n \in \Bz$ is well defined and is a Borel isomorphism from $X(\cd)$ to $\natext_M$.
\label{isothm}
\end{theorem}
\noindent Hence, one could say that $X(\cd)$ is ``partially isomorphic" to $\natext.$

 It is apparent that the HB diagram of a Sturmian system $X_u^+$ does not contain any bi-infinite paths, thus $X(\cd)_u$ is the empty set. This may seem alarming, but it turns out that the eventually Markov part of $\natext_u$ is also empty. In fact, we show that if the natural extension of a subshift is infinite and minimal, then the eventually Markov part of the natural extension is empty. Thus, it will follow that if $X^+$ is minimal the isomorphism in Theorem \ref{isothm} is between two copies of the empty set.
 {Nevertheless, Theorem \ref{pathprop} gives an isomorphism between $X(\cd)^+$ (paths in the HB diagram that start with blocks of length one, see Definition \ref{def:one-sidedpaths}) and $X^+$.}

\begin{prop}
If $X^+$ is a subshift  such that $\natext$ is infinite and minimal, then the eventually Markov part of $\natext$ is empty.
\label{evMark}
\end{prop}

\begin{proof}
Suppose on the contrary that there exists $x \in \natext$ that is eventually Markov at time $p \in \Bz$.
Then there exists $N=N(x,p)$ such that for all $n \geq N$,
\begin{center}
$\fol(x_{p-n}...x_p) = \fol(x_{p-N}...x_p).$
\end{center}
Let $B= x_{p-N}...x_p.$
%Let $B = a_{p-N}...a_p$.
%Since $X^+$ is minimal, it follows that every block is left extendable, hence $\cl(X^+) = \cl(\natext).$ Furthermore, every block  in $\cl(a)$ appears with bounded gap to the left and right $a$.
Since $x$ is left recurrent there exists $n > 2N+1$ such that $x_{p-n}...x_{p-n+N} = B$. Let $A = x_{p-n+N+1}...x_{p-N-1}$.

By the definition of $\natext$, $x_{p-n}x_{p-n+1}...=BABx_{p+1}x_{p+2}... \in X^+$. Thus there exists a ray $r_1$ in $\fol(B)$ that has $x_{p-n+N}AB = x_pAB$ as a prefix. Since
\begin{center}
$\fol(B) =\fol(x_{p-N}...x_p)= \fol(x_{p-n}...x_p) = \fol(BAB)$,
\end{center}
 it follows that $r_1 \in \fol(BAB).$ This implies that there exists an $a^{(1)} \in X^+$ with prefix $BABAB$.

Since $a^{(1)} \in X^+$, there exists a ray $r_2$ in $\fol(B)$ that has $a_pABAB$ as a prefix. Then $\fol(B) = \fol(BAB)$ implies that $r_2 \in \fol(BAB)$. Hence there exists $a^{(2)} \in X^+$ with prefix $BABABAB$. Continuing in this manner, we construct a sequence $(a^{(n)}) \subset X^+$ with
\begin{center}
$\displaystyle \lim_{n \to \infty}(a^{(n)}) = BABABABABA...$.
\end{center}
Let $b^{(n)}=0^{\infty}.a^{(n)}$ and $x_n(a^{(n)})= \sigma^nb^{(n)}$ as in Proposition \ref{extcons}. Then any limit point $y$ of $(x_n(a^{(n)}))$ is a periodic bisequence in $\natext$. This is a contradiction, since $\natext$ does not contain any periodic points. Thus $x$ is not eventually Markov at any time $p \in \Bz$ and $\natext_M$, the eventually Markov part of $\natext$, is empty.
\end{proof}

As previously discussed, a Sturmian sequence $u$ is syndetically recurrent and is not periodic. Since $X_u^+$ is the orbit closure of the almost periodic sequence $u$ it follows that $X_u^+$ is minimal. Furthermore, since $u$ is not periodic, $X_u^+$ is infinite. However, a priori, we don't know that $\natext_u$ is minimal.

\begin{prop}
If $X^+$ is minimal, then the natural extension $\natext$ of $X^+$ is minimal. In fact, for any $x \in \natext$ both the forward orbit $\co^+(x) = \{\sigma^nx | n \geq 0 \}$ and the backward orbit $\co^{-}(x) = \{\sigma^{-n}x | n \geq 0 \}$ are dense in $\natext$. %Furthermore, if $X^+$ is also infinite, then $\natext$ is infinite.
\label{inf}
\end{prop}

\begin{proof}
Since $X^+$ is minimal, every block $B \in \cl(X^+)$ appears with bounded gap in each $a \in X^+$. By Corollary \ref{corleftext}, $\cl(X^+) = \cl(\natext).$ Therefore each block $B \in \cl(\natext)$ appears in each long-enough block in $\cl(\natext).$ Hence for all $x \in \natext$, the block $B$ appears with bounded gap to the left and the right in $x$.  Thus $\co^+(x)$ and $\co^{-}(x)$ are dense in $\natext$.
\end{proof}

\begin{remark}
Since any infinite minimal subshift contains no periodic points, it follows from Proposition \ref{inf} that if $X^+$ is both minimal and infinite, then $\natext$ is minimal and infinite.
\end{remark}

\section{The Morse minimal subshift} \label{Morse}
To describe the construction of the HB diagram of one particular substitution system, the Morse minimal subshift, we have to recall some well-known properties of the Prouhet-Thue-Morse sequence,
\begin{center}
$\omega = .\omega_0\omega_1\omega_2... = .0110100110010110....$
\end{center}
For the sake of simplicity, we shall refer to this sequence as the Morse sequence. The one-sided subshift associated with $\omega$ is the \emph{Morse minimal subshift}. It is defined by the pair $(X_{\omega}^+, \sigma)$, where $X_{\omega}^+$ is the closure of $\{\sigma^n\omega | n \in \Bn \}$.

This sequence has many interesting properties. Axel Thue, concerned with constructing bi-infinite sequences on two symbols with controlled repetitions, constructed the two-sided Morse sequence
\begin{center}
$M = ...0110100110010110.0110100110010110...,$
\end{center}
which he defined as having the property that the sequence contains no blocks of the form $BBb$ where $B$ is a block and $b$ is the first letter of $B$ \cite{Coven, Fife}. Thue's results were published in 1912.  In 1917, Marston Morse, not knowing of Thue's results, constructed the Morse sequence in his dissertation. In \cite{Unendingchess} Morse and Hedlund proved that every element in the Morse minimal set, the closure of $\{\sigma^n(M) | n \in \Bz\}$, has the no $BBb$ property. It was later shown by Gottshalk and Hedlund that the elements of the Morse minimal set are the only bi-infinite sequences with the no $BBb$ property \cite{OnlynoBBb}. While the Gottshalk and Hedlund result does not carry over to the one-sided Morse sequence  \cite{Fife}, it is still the case that the one-sided Morse sequence $\omega$ has the no $BBb$ property.

The Morse sequence is also generated by iterating a substitution. Following Chapter 5 of \cite{Fogg}, we recall how this is done and how the construction allows us to deduce important properties of the sequence. Let $\zeta$ be the substitution map defined by $\zeta(0) = 01$ and $\zeta(1) = 10$. The Morse sequence is the infinite sequence which begins with $\zeta^n(0)$ for every $n \in \mathbb{N}$. It follows from this construction that the Morse sequence is syndetically recurrent and neither periodic nor eventually periodic.

\subsection{Recognizability of the Morse substitution}
Since the Morse sequence arises from a substitution map, it is natural to consider how to ``decompose" or ``desubstitute" a block that occurs in $\cl(X_\omega^+).$ The notion of recognizability deals with this problem \cite{Quef}.

\begin{definition} A substitution $\gamma$ over the alphabet $\ca$ is \emph{primitive} if there exists $k \in \Bn$ such that for all $a, b \in \ca$ the letter $a$ occurs in $\gamma^k(b).$
\end{definition}
\noindent In the context of recognizability we consider only primitive substitutions. Note that the Morse substitution $\zeta$ is primitive since 0 and 1 both appear in $\zeta(0)$ and $\zeta(1)$.

Let $u = u_0u_1...$ be any fixed point of an aribitrary primitive substitution $\gamma$.
\begin{definition}
For every $k \geq 1$, $\displaystyle E_k = \{0\} \cup \{|\gamma^k(u_0u_1...u_{p-1})| \hspace{2mm} | \hspace{2mm} p > 0\}$ is the set of \emph{cutting bars of order k}.
\end{definition}

\begin{definition}
The substitution $\gamma$ is said to be \emph{recognizable} if there exists an integer $K > 0$ such that
\begin{center}
$n \in E_1$ and $u_nu_{n+1}...u_{n+K} = u_{m}u_{m+1}...u_{m+K}$ implies $m \in E_1$.
\end{center}
The smallest integer $K$ satisfying this is the \emph{recognizability index} of $\gamma.$
\end{definition}
In other words, a substitution is recognizable if it is possible to determine if $u_m$ is the first letter of a substituted block by examining the $K$ terms that follow it.  The Morse substitution $\zeta$ is recognizable with recognizability index $3$. This means that it is possible to determine if 0 (or 1) is the first letter of $\zeta(0)$ (or $\zeta(1)$) by examining the three letters which follow it.

Note that this definition of recognizability does not satisfactorily guarantee desubstitution in the general setting. Even very simple primitive, aperiodic substitutions may fail to have the recognizability property. For example, the substitution $\gamma$ on the alphabet $\{0,1 \}$ defined by $\gamma(0)=010$ and $\gamma(1)=10$ fails to be recognizable. Brigitte Moss\'e introduces another notion of recognizability, bilateral recognizability, in \cite{Mosse}.
\begin{definition}
A substitution $\gamma$ is said to be \emph{bilaterally recognizable} if there exists an integer $L>0$ such that
\begin{center}
$n\in E_1$ and $u_{n-L}...u_{n+L} = u_{m-L}...u_{m+L}$ implies $m \in E_1$.
\end{center}
\end{definition}
\noindent One advantage of Moss\'e's definition is that every primitive aperiodic substitution is bilaterally recognizable. Furthermore, if
 $u$ is a fixed point of a primitive aperiodic substitution $\gamma$ and
  {$X_u^+ = \text{cl}{\{\sigma^ku| k \in \Bn\}}$ (cl denotes closure)},
 then any block in $\cl(X_u^+)$ can be ``desubstituted" up to some prefix and some suffix at the ends of the block \cite{Mosse, Quef, Fogg}. Since the Morse substitution is recognizable, we do not rely on bilateral recognizability. However, the consequences of bilateral recognizability could be useful in extending the results for the Morse minimal subshift to general substitution systems.

We now consider the decomposition of blocks appearing in the fixed point of a substitution. Let $b= u_i...u_{i+|w|-1}$ be a block appearing in $u$. Since $\gamma(u) = u$ there exists an index $j$, a length $l$, a suffix $S$ of $\gamma(u_j)$ and a prefix $P$ of $\gamma(u_{j+l+1})$ such that
\begin{center}
$b = S\gamma(u_j+1)...\gamma(u_{j+l})P.$
\end{center}
\begin{definition}
Let $b$ be as above. The \emph{1-cutting at the index i} of $b$ is
\begin{center}
$S \dagger \gamma(u_{j+1}) \dagger... \dagger \gamma(u_{j+l}) \dagger P$,
\end{center}
and we say that $b$ comes from the block $u_j...u_{j+l+1}$. The block $u_j...u_{j+l+1}$ is the \emph{ancestor block} of $b$ \cite{Fogg}.
\end{definition}
Note that $S$ and $P$ are not necessarily proper suffixes and prefixes, respectively. Furthermore, the 1-cutting yields a string on an enlarged alphabet. For the Morse sequence this alphabet is $\{0,1,\dagger \}.$

To illustrate this, consider the block $\omega_4...\omega_9 = 1001100$ appearing in the Morse sequence. Let $S = 10 = \zeta(\omega_2)$, and $P = 0$, the one letter prefix of $\zeta(\omega_5)$. Then
\begin{center}
$1001100 = 10 \dagger 01 \dagger 10 \dagger 0 = 10 \dagger \zeta(\omega_3) \dagger \zeta(\omega_4) \dagger 0$,
\end{center}
and has $\omega_2...\omega_5 = 1010$ as an ancestor block.

In this example, it is apparent that the 1-cutting of the block 1001100 partitions the block into a concatenation of the subblocks $10$ and $01$ with daggers in between.
We define the \emph{1-blocks} of the Morse sequence to be the blocks $01$ and $10$. By partitioning a block into its 1-blocks, it is possible to determine its ancestor block. The following lemma, found in \cite{Fogg}, is a result of the recognizability of the Morse sequence.

\begin{lemma}
In the Morse sequence, every block of length at least five has a unique 1-cutting, or decomposition into 1-blocks, possibly beginning with the last letter of a 1-block and possibly ending with the first letter of a 1-block.
\label{1word}
\end{lemma}

\begin{remark}
If a block has a unique 1-cutting, then the block has a unique ancestor block. The only blocks of length less than five appearing in the Morse sequence which do not have a unique partition into 1-blocks are 010, 101, 0101, and 1010, each of which has two possible ancestor blocks. Furthermore, a block has a unique 1-cutting if and only if that block has either 00 or 11 as a subblock.
\label{smallword}
\end{remark}

We denote the \emph{dual} of a letter $a \in \{0,1\}$ by $\overline{a}$. If $a = 0$ then $\overline{a} = 1$ and vice versa. Note that each 1-block consists of a pair of dual letters.
\begin{lemma}
Let $a_{-n}a_{-n+1}a_{-n+2}...a_{0}$ be a block of length $n$ in $\cl(X_{\omega}^+)$ that has a unique 1-cutting. If the unique 1-cutting has a dagger immediately to the right of $a_{-n+1}$, then $a_{-n} = \overline{a}_{-n+1}$. That is, $a_{-n}$  is uniquely determined by $a_{-n+1}.$
\label{barlem}
\end{lemma}

\begin{proof}
Since $a_{-n+1}a_{-n+2}...a_{0}$ has a unique decomposition into 1-blocks, $a_{-n+1}a_{-n+2}...a_{0}$ has $00$ or $11$ as a subblock. Without loss of generality, suppose $00$ is a subblock of $a_{-n+1}a_{-n+2}...a_{0}$. Since $00$ is not a 1-block, it follows that in the 1-cutting of $a_{-n+1}a_{-n+2}...a_{0}$ there is a dagger in the middle of $00$. Furthermore, since there is a dagger immediately to the right of $a_{-n+1}$, it follows that there exists $k$, $1\leq k \leq n/2$, such that $a_{-n-1+2k}a_{-n+2k} = 00$.

Now consider $a_{-n}a_{-n+1}...a_{0} \in \cl(X_{\omega}^+)$. Since $00$ is a subblock of $a_{-n}a_{-n+1}...a_{0}$, there exists a unique 1-cutting of $a_{-n}a_{-n+1}a_{-n+2}...a_{0}$.
Additionally, there must be a dagger in between $00 = a_{-n-1+2k}a_{-n+2k}$ in the unique 1-cutting. Thus there is a dagger between $a_{-n-1+2i}a_{-n+2i}$ for all $1\leq i \leq n/2.$ Letting $i=1$, this implies that there is a dagger  between $a_{-n+1}a_{-n+2}$. Hence,  $a_{-n}a_{-n+1}$ must be a 1-block and  $a_{-n} = \overline{a_{-n+1}.}$
\end{proof}

\subsection{Significant blocks of the Morse minimal subshift}
Using the properties of the Morse sequence detailed above, we determine the significant blocks of the Morse minimal subshift. Note that it follows from the minimality of the Morse sequence that $\cl(X_{\omega}^+) = \cl(\natext_{\omega}).$

\begin{prop} Let $a_{-n+1}a_{-n+2}...a_{-2}a_{-1}a_{0}$ be a block of length $n \in \mathbb{N}$, $n \geq2$, in $\cl(X_{\omega}^+)$. Then $a_{-n+1}a_{-n+2}...a_{-2}a_{-1}a_{0}$ is significant if and only if $0a_{-n+2}...a_{-2}a_{-1}a_{0}$ and $1a_{-n+2}...a_{-2}a_{-1}a_{0}$ are in $\cl(X_{\omega}^+)$.
\label{Morsesig}
\end{prop}

\begin{proof}
One direction is proved in Lemma \ref{leftextend}.

We first prove the converse for 01,10, 010, 101, 0101, and 1010. By examining the sequence $\omega$ it is apparent that each of these blocks satisfies the hypothesis that $0a_{-n+2}...a_{-2}a_{-1}a_{0}$ and $1a_{-n+2}...a_{-2}a_{-1}a_{0}$ are in $\cl(X_{\omega}^+)$. To prove that each block is significant we construct a ray for each block that is in $\fol(a_{-n+2}...a_{-2}a_{-1}a_{0})$, but not in $\fol(a_{-n+1}a_{-n+2}...a_{-2}a_{-1}a_{0}).$

Consider $010$. Since the Morse sequence contains no blocks of the form $BBb$, $01010 \notin \cl(X_{\omega}^+)$. However, $1010 \in \cl(X_{\omega}^+).$ Thus the ray $\omega_3\omega_4\omega_5... = 0100110...$ is in $\fol(10)$ but not $\fol(010).$ Hence $010$ is significant. Similarly, it can be shown that

\begin{center}

$\omega_{2}\omega_{3}\omega_{4}... = 101001... \in \fol(1)$ but  $101001...  \notin \fol(01)$,

$\omega_{10}\omega_{11}\omega_{12}... = 010110... \in \fol(0)$ but  $010110...  \notin \fol(10)$,

$\omega_{11}\omega_{12}\omega_{13}... = 101101... \in \fol(01)$ but  $101101...  \notin \fol(101)$,

$\omega_{4}\omega_{5}\omega_{6}... = 100110... \in \fol(101)$ but $100110... \notin \fol(0101)$,

 $\omega_{12}\omega_{13}\omega_{14}...=011010...  \in \fol(010)$ but $011010...  \notin \fol(1010).$
 \end{center}
Therefore, 01, 10, 101, 0101, and 1010 are all significant.

We now prove the converse for the remaining blocks $a_{-n+1}a_{-n+2}...a_{-2}a_{-1}a_{0}$ satisfying the property that $0a_{-n+2}...a_{-2}a_{-1}a_{0}$ and $1a_{-n+2}...a_{-2}a_{-1}a_{0}$ are in $\cl(X_{\omega}^+)$. To prove that each of these blocks is significant, we explicitly construct a ray that is in the follower set of $a_{-n+2}...a_{-2}a_{-1}a_{0}$ but not in the follower set of $a_{-n+1}a_{-n+2}...a_{-2}a_{-1}a_{0}$. This is done by repeatedly desubstituting $a_{-n+1}a_{-n+2}...a_{-2}a_{-1}a_{0}$ and choosing a ray based on the ancestor block of $a_{-n+1}a_{-n+2}...a_{-2}a_{-1}a_{0}$.

It follows from Lemma \ref{1word} and Remark \ref{smallword} that each of the remaining blocks has a unique 1-block decomposition.  If $n >3$, then partition $a_{-n+1}a_{-n+2}...a_{-2}a_{-1}a_{0}$ into 1-blocks.
Lemma \ref{barlem} implies that if the unique partition of $a_{-n+1}a_{-n+2}...a_{-2}a_{-1}a_{0}$ has a dagger directly after $a_{-n+2}$, then $a_{-n+1}$ is uniquely determined by $a_{-n+2}$, and hence $0a_{-n+2}...a_{-2}a_{-1}a_{0}$ and $1a_{-n+2}...a_{-2}a_{-1}a_{0}$ are not both in $\cl(X_{\omega})$. It follows that for $n>3$ the partition of $a_{-n+1}a_{-n+2}...a_{-2}a_{-1}a_{0}$ is
\begin{center} $a_{-n+1}\dagger a_{-n+2}a_{-n+3} \dagger \cdots \dagger a_{-2}a_{-1}\dagger a_{0}$ \end{center}
if $n$ is even, and
\begin{center} $a_{-n+1} \dagger a_{-n+2}a_{-n+3} \dagger \cdots \dagger a_{-3}a_{-2} \dagger a_{-1}a_{0}$ \end{center} if $n$ is odd.

Next we map each 1-block in the partition to its preimage under $\zeta$, or ancestor block.
Let $\displaystyle s_{1,i/2}$ denote the preimage of the 1-block $a_{-n+i}a_{-n+i+1},$ where $i$ is even and $2 \leq i \leq n-1.$
Since there is a dagger placed directly after $a_{-n+1}$, $a_{-n+1}$ uniquely determines the letter that can precede it. Let $a_{-n}$ be this letter and let  $s_{1,0}$ denote the preimage of $a_{-n}a_{-n+1}$. Similarly, when $n$ is even $a_0$ uniquely determines $a_1$. In this case let $s_{1,n/2}$ denote the preimage of $a_0a_1.$  The resulting block is $\displaystyle s_{1,0}s_{1,1}...s_{1,(n-1)/2}$ if $n$ is odd, and $\displaystyle s_{1,0}s_{1,1}...s_{1, n/2}$ if $n$ is even. For ease of notation, denote $\displaystyle s_{1,0}s_{1,1}...s_{1,(n-1)/2}$ or $\displaystyle s_{1,0}s_{1,1}...s_{1, n/2}$ by $S_1$. Note that $S_1$ is the ancestor block of $a_{-n+1}a_{-n+2}...a_{-2}a_{-1}a_{0}$.

At this stage, consider the length of  $S_1$. If $|S_1| \geq 4$ and $S_1$ has a unique 1-block partition, then map each 1-block to its preimage under $\zeta$. That is, map $S_1$ to its ancestor block.
If $|S_1| <4$ or $S_1 = 0101$ or $1010$, then do nothing. We continue this process so that in general, if $|S_j| \geq 4$ and $S_j$ can be uniquely partitioned into 1-blocks, we map $S_j$ to its ancestor block,
and otherwise do nothing. During this process, if the 1-block decomposition of $S_j$ has a dagger between $s_{j,0}$ and $s_{j,1}$, then there is only one letter that can precede $s_{j,0}$. Let $s_{j,-1}$ be this letter. Then $s_{j+1,0}$ is defined to be the preimage of the 1-block $s_{j,-1}s_{j,0}.$
\newline

\begin{example}
 We illustrate this process for the block $00110100 \in \cl(X_{\omega}^+).$ The unique 1-block decomposition of $00110100$ is
\begin{center}
$\displaystyle 0 \dagger 01 \dagger 10 \dagger 10 \dagger 0$.
\end{center}
Mapping each 1-block to its preimage under $\zeta$, we get $S_1= 10110$. Since $|10110| = 5$, we partition $10110$ into 1-blocks as follows:
\begin{center}
$\displaystyle 1\dagger 01 \dagger 10$.
\end{center}
Again, map each 1-block to its preimage under $\zeta$ to get $S_2 = 001$. As $|001| < 4$, the decomposition process is complete.
\newline
\end{example}

We claim that there exists an $m \in \mathbb{N}$ such that $S_m = s_{m,0}s_{m,1}s_{m,2}$, where $s_{m,1}s_{m,2} \in \{01,10 \}$; or $S_m = s_{m,0}s_{m,1}s_{m,2}s_{m,3}$, where  $s_{m,1}s_{m,2}s_{m,3} \in \{010,101\}$. Since the decomposition process is repeated until $3 \leq |S_m| \leq 4$ we need only show that
\begin{center}
$S_m \in \{001, 101, 010, 110, 0010,1010, 0101, 1101\}$.
\end{center}
Suppose on the contrary that $S_m$ is not one of the above blocks. Since $S_m \in \cl(X_{\omega}^+)$, it follows that
\begin{center}
$S_m \in \{100, 011, 1001,0110\}$.
\end{center}
Each of these blocks has a unique partitioning in which there is a dagger directly to the right of $s_{m,1}$. Hence $s_{m,0}$ is uniquely determined by ${s_{m,1}}$, and $s_{m,0} = \overline{s_{m,1}}$, by Lemma \ref{barlem}. We prove that if $s_{m,0}$ is uniquely determined by $s_{m,1}$ then there is only one possible value for $a_{-n+1}$.

First suppose that for all $1\leq k < m$ the partition of $S_k$ has a dagger between $s_{k,0}$ and $s_{k,1}$. That is,
\begin{center}
$\displaystyle s_{k,0} \dagger s_{k,1}s_{k,2} \dagger \cdots \dagger s_{k,j}$ \hspace{3mm} and \hspace{3mm} $\displaystyle s_{k,0} \dagger s_{k,1}s_{k,2} \dagger \cdots \dagger s_{k,j}s_{k, j+1}$
\end{center}
are the 1-cuttings for $1 \leq k < m$ when $|S_k|$ is even and odd respectively, but
\begin{center}
$\displaystyle s_{m,0} s_{m,1} \dagger s_{m,2}$ \hspace{3mm} and  \hspace{3mm} $\displaystyle s_{m,0} s_{m,1} \dagger s_{m,2}s_{m,3}$
\end{center}
are the 1-cuttings for $s_{m,0}$ when $S_m$ has length 3 and 4 respectively.

If this is the case, then $\zeta(s_{k+1,0}) = s_{k, -1}s_{k,0}$ for all $1\leq k < m$.
Then
\begin{center}
$\zeta^2(s_{k+1,0})= \zeta(s_{k,-1})\zeta(s_{k,0}) =  \zeta(s_{k,-1})s_{k-1,-1}s_{k-1,0},$
\end{center}
and in general
\begin{center}
$\zeta^j(s_{k+1,0})= \zeta^{j-1}(s_{k,-1})\zeta^{j-2}(s_{k-1,-1})...\zeta(s_{k-j+2,-1})s_{k-j+1,-1}s_{k-j+1,0}$
\end{center}
for $j \leq k$.
Hence
\begin{center}
\begin{align*}
\zeta^m(s_{m,0}) &= \zeta^{m-1}(s_{m-1,-1})\zeta^{m-2}(s_{m-2,-1})...\zeta^2(s_{2,-1})\zeta(s_{1,-1})\zeta(s_1,0) \\
& = \zeta^{m-1}(s_{m-1,-1})\zeta^{m-2}(s_{m-2,-1})...\zeta^2(s_{2,-1})\zeta(s_{1,-1})\overline{a_{-n+1}}a_{-n+1}.
\end{align*}
\end{center}
Thus the last letter of $\zeta^m(s_{m,0})$ is $a_{-n+1}$. Since each desubstitution is unique, this implies that $a_{-n+1}$ is uniquely determined by $s_{m,0}.$

Let $S_m = s_{m,0}...s_{m,q}$, where $q \in \{2,3\}$. By a similar argument to that used above, it can be shown that the block $a_{-n+2}a_{-n+3}..a_1a_0$ is a prefix of $\zeta^m(s_{m,1}...s_{m,q})$. That is,
\begin{center}
$\zeta^m(S_m) = \zeta^m(s_{m,1}...s_{m,q}) = a_{-n+2}a_{-n+3}...a_1a_0C$,
\end{center}
where $C \in \cl(X^+_\omega).$ Since each desubstitution of $a_{-n+2}a_{-n+3}...a_1a_0$ is unique, this implies that  $s_{m,1}$ is uniquely determined  by the block $a_{-n+2}a_{-n+3}...a_1a_0$.

Furthermore, since we have assumed that $s_{m,0}$ is uniquely determined by $s_{m,1}$, it follows that $\zeta^m(s_{m,0}) = \zeta^m(\overline{s_{m,1}})$. Hence, $a_{-n+1}$, the last letter of $\zeta^m(s_{m,0})$, is uniquely determined by $s_{m,1}$. Thus, there is only one possible value for $a_{-n+1}$.
Therefore,  $0a_{-n+2}a_{-n+3}...a_1a_0$ and $1a_{-n+2}a_{-n+3}...a_1a_0$ cannot both be in $\cl(X_{\omega}^+)$, contradicting $a_{-n+1}a_{-n+2}a_{-n+3}...a_1a_0$ being a significant block.

Now suppose that there exists  $1\leq k < m$ such that the 1-block decomposition of $S_k$ has a dagger after $s_{k,1}$. Then there exists an $r$, $1\leq r \leq k$, such that for all $1\leq j < r$ the partition of $S_j$ has a dagger between $s_{j,0}$ and $s_{j,1}$, but the partition of $S_r$ has a dagger after $s_{r,1}$. By the previous argument, it follows that $a_{-n+1}$ is uniquely determined by $s_{r,0}$ and thus by $s_{r,1}$. Hence, $0a_{-n+2}a_{-n+3}..a_1a_0$ and $1a_{-n+2}a_{-n+3}..a_1a_0$ cannot both be in $\cl(X_{\omega}^+)$. Therefore,
\begin{center}
$S_m \notin \{100, 011, 1001,0110\}$.
\end{center}

\begin{remark}
Let $S_m = s_{m,0}...s_{m,q}$, where $q \in \{2,3\}$. By the above argument it must be the case that for all $1\leq k < m$ the partition of $S_k$ has a dagger between $s_{k,0}$ and $s_{k,1}$. This implies that $a_{-n+1}$ is the last letter of $\zeta^m(s_{m,0})$ and $a_{-n+2}a_{-n+3}..a_1a_0$ is a prefix of $\zeta^m(s_{m,1}...s_{m,q})$, as discussed.
\end{remark}

Having established the existence of an $m \in \mathbb{N}$ such that $S_m = s_{m,0}s_{m,1}s_{m,2}$, where $s_{m,1}s_{m,2} \in \{01,10 \}$, or $S_m = s_{m,0}s_{m,1}s_{m,2}s_{m,3}$, where  $s_{m,1}s_{m,2}s_{m,3} \in \{010,101\}$, define the ray $\nu$ as follows.

\noindent If $s_{m,0} = 0$ then,
$$
\nu =
\begin{cases}
 \omega_5\omega_6\omega_7...  = 0011001... ,& \text{if }s_{m,1}s_{m,2} = 01 \\
 \omega_6\omega_7\omega_8...  = 0110010... ,& \text{if }s_{m,1}s_{m,2s_{m,3}} = 010\\
  \omega_4\omega_5\omega_6... =  1001100... ,& \text{if }s_{m,1}s_{m,2} = 10 \\
  \omega_5\omega_6\omega_7... = 0011001...  ,& \text{if }s_{m,1}s_{m,2s_{m,3}} = 101 \\
\end{cases}
$$
If $s_{m,0} = 1$ then,
$$
\nu =
\begin{cases}
\omega_{13}\omega_{14}\omega_{15}... = 1101001... , & \text{if }s_{m,1}s_{m,2} = 10 \\
\omega_{14}\omega_{15}\omega_{16}... = 1010010... ,& \text{if }s_{m,1}s_{m,2s_{m,3}} = 101\\
\omega_{12}\omega_{13}\omega_{14}... = 0110100... ,& \text{if }s_{m,1}s_{m,2} = 01 \\
\omega_{13}\omega_{14}\omega_{15}... = 1101001... ,& \text{if }s_{m,1}s_{m,2s_m,3} = 010 \\
\end{cases}
$$
Note that the sequence $\nu$ is defined so that $s_{m,1}s_{m,2}\nu$  (or $s_{m,1}s_{m,2}s_{m,3}\nu$) is in $X_{\omega}^+$, but
$s_{m,0}s_{m,1}s_{m,2}\nu$ (or $s_{m,0}s_{m,1}s_{m,2}s_{m,3}\nu$) is not in $X_{\omega}^+$. We provide an example of this.

\begin{example} Let $s_{m,0}s_{m,1}s_{m,2} = 001$. Then $s_{m,0}s_{m,1}s_{m,2}\nu = 0010011001...$. Consider the 1-block decomposition of the prefix block $00100$. Since a dagger must be placed between consecutive zeros we get $0 \dagger 01 \dagger 0 \dagger 0$. This, however, is not an allowed 1-block decomposition. Thus $00100$ is not in  $\cl(X_{\omega}^+)$ and $s_{m,0}s_{m,1}s_{m,2}\nu \notin X_\omega^+$.
\end{example}

Suppose $s_{m,1}s_{m,2} \in \{01,10\}$. Let $\zeta^{m}(s_{m,1}s_{m,2}\nu) = d_1d_2d_3...$.
Note that the first $n-1$ letters of $\zeta^{m}(s_{m,1}s_{m,2}\nu)$ form the block $a_{-n+2}a_{-n+3}...a_0$. We claim that $d_{n-1}d_{n}d_{n+1}...$ is in $\text{fol}(a_{-n+2}a_{-n+3}...a_0)$ but not in $\text{fol} (a_{-n+1}a_{-n+2}...a_0)$.
Since $s_{m,1}s_{m,2}\nu \in X_{\omega}^+$, it follows that $\zeta^{m}(s_{m,1}s_{m,2}\nu) \in X_{\omega}^+$, as $\omega$ is fixed under the substitution $\zeta$. Thus, $d_{n-1}d_{n}d_{n+1}...$ is in $\text{fol}(a_{-n+2}a_{-n+3}...a_0)$. It remains to show that $a_{-n+1}\zeta^{m}(s_{m,1}s_{m,2}\nu)$ is not in $X_{\omega}^+$.

Let $\nu=\nu_1\nu_2\nu_3...$ and consider the block $a_{-n+1}\zeta^m(s_{m,1}s_{m,2}\nu_1\nu_2).$ Since the first $n-1$ terms of $\zeta^{m}(s_{m,1}s_{m,2}\nu_1\nu_2)$ form the block $a_{-n+2}a_{-n+3}...a_0$, the block $a_{-n+1}a_{-n+2} a_{-n+3}...a_0$ is a prefix of the block $a_{-n+1}\zeta^m(s_{m,1}s_{m,2}\nu_1\nu_2).$ Implementing the decomposition process $m$ times, the resulting block is  $s_{m,0}s_{m,1}s_{m,2}\nu_1\nu_2$. However, $s_{m,0}s_{m,1}s_{m,2}\nu_1\nu_2 \notin \cl(X_{\omega}^+).$  As the ancestor block of any block in $\cl(X_{\omega}^+)$ is in $\cl(X_{\omega}^+)$ it follows that $a_{-n+1}\zeta^m(s_{m,1}s_{m,2}\nu_1\nu_2) \notin \cl(X_{\omega}^+)$.
Thus, $a_{-n+1}\zeta^{m}(s_{m,1}s_{m,2}\nu)$ is not in $X_{\omega}^+$. Therefore, $a_{-n+1}a_{-n+2}a_{-n+3}...a_0$ is significant.

Using similar arguments, it can be shown that $a_{-n+1}a_{-n+2}...a_0$ is significant in the case that
\begin{center}
$s_{m,1}s_{m,2}s_{m,3} \in \{010,101\}$.
\end{center}
\end{proof}

\begin{remark}
We relate this result to Proposition \ref{Prop1} from the Sturmian case. Since there is not a unique left special sequence as in the Sturmian case, let $m_n$ be the number of blocks of length $n$ in $\cl(X_\omega^+)$ that can be extended to the left in two ways. Denote the set of such blocks by $\{L_n^i\}$, $1 \leq i \leq m_n$. Then for each $n \geq 2$ the significant blocks of length $n$ of $X_\omega^+$ are $\{0L_{n-1}^i, 1L_{n-1}^i\}$, $1 \leq i \leq m_{n-1}$.
\end{remark}

\subsection{HB diagram of the Morse minimal subshift}
Since there is no left special sequence to direct us to the blocks in $\cl(X_\omega^+)$ that can be extended to the left in two ways, the process of determining the significant blocks is more tedious. Nevertheless, it is possible to use the no $BBb$ property and desubstitution to identify the significant blocks.

After determining the significant blocks of the Morse minimal subshift, the next step is to determine the arrows. As in the Sturmian case, only those blocks that can be extended to the right in two ways will have two arrows out. However, for the Morse minimal subshift there is no easy technique for determining the significant blocks that satisfy this property. Hence the process of constructing the HB diagram of the Morse minimal subshift is not nearly as streamlined as for the Sturmians. We construct the HB diagram in the following way.

Begin by generating a list of signficant blocks. Start with the blocks $0$ and $1$.  Since $00, 10, 01$, and $11$ are all in $\cl(X_\omega^+)$, it follows that $0$ and $1$ can both be extended to the left in two ways. Hence $00, 10, 01$, and $11$ are all significant. Next consider $\cl_2(X_\omega^+) = \{00,01,10,11 \}$, the blocks of length 2 in $\cl(X_\omega^+)$. As the Morse minimal subshift has no blocks of the form $BBb$, $000$ and $111$ are not in $\cl(X_\omega^+)$. Hence $00$ and $11$ cannot be extended to the left in two ways. However, $10$ and $01$ can be extended to the left in two ways. Thus $110, 010, 001,$ and $101$ are the significant blocks of length 3.

Now consider $\cl_3(X_\omega^+) = \{001, 010, 011, 100, 101, 110\}$. By the no $BBb$ property, it follows that $0001$ and $1110$ are not in $\cl(X_\omega^+)$. The remaining blocks can all be extended to the left in two ways. Thus the significant blocks of length 4 are 0010, 1010, 0011, 1011, 0100, 1100, 0101,  and $1101.$  Continuing in this manner, we are able to generate a list of significant blocks of $X_\omega^+$;
\begin{center}
0,1, 00, 10, 01, 11, 110, 010, 001, 101, 0010, 1010, 0011, 1011, 0100, 1100, 0101, 1101, 00110, 01001,10110, 11010,... .
\end{center}

To determine the arrows in the HB diagram of $X_\omega^+$, we consider the right extensions of each significant block. We illustrate the process of determining the arrows by considering those arrows that start at a significant block of length 4. It is easily seen that $0011, 1011, 0100,$ and $1100$ can only be extended to the right in one way. Thus there is exactly one arrow out of each of these blocks, and these arrows are:
\begin{center}
$0011 \to \sig(00110) = 00110$ \\
$1011 \to \sig(10110) = 10110$ \\
$0100 \to \sig(01001) = 01001$ \\
$1100 \to \sig(11001) = 11001. $
\end{center}
Additionally, $1010$ and $0101$ can be extended to the right in only one way, as the blocks $10101$ and $01010$ are of the form $BBb$. The arrows out of these blocks are:
\begin{center}
$1010 \to \sig(10100) = 0100$ \\
$0101 \to \sig(01011) = 1011$.
\end{center}
Lastly, consider the blocks $1101$ and $0010$. Instead of using the no $BBb$ property, we consider the 1-block decomposition of each block. This gives us $1\dagger10\dagger1$ and $0\dagger01\dagger0$. Since extending each block to the right must yield a legal 1-block decomposition, it follows that $1101$ can be followed only by a 0, and $0010$ can be followed only by a 1. Thus, the arrows out of these blocks are:
\begin{center}
$1101 \to \sig(11010) = 1010$ \\
$0010 \to \sig(00101) = 0101$.
\end{center}
Note that although there is exactly one arrow out of each significant block of length 4, this is not the case in general.
For example, using the same process it can be shown that all four significant blocks of length 5 can be extended to the right in two ways.

Figure \ref{fig:MMS} depicts a portion of the HB diagram of the Morse minimal subshift that has been constructed using the process described above.

\begin{figure}[h]
\begin{center}
%\scalebox{.6}{
\begin{tikzpicture}[node distance=2.5cm,auto,>=latex']
\tiny{
    \node (a1) {};
    \node(a2)  [right of= a1, node distance=1.5cm]{};
    \node(a3) [right of= a2, node distance=1.5cm]{};
    \node(a4) [right of= a3, node distance=1.5cm]{0};
    \node(a5) [right of= a4, node distance=1.5cm]{1};
    \node(a6) [right of= a5, node distance=1.5cm]{};
    \node(a7) [right of= a6, node distance=1.5cm]{};
    \node(a8) [right of= a7, node distance=1.5cm]{};

    \node (b1) [below of=a1, node distance=1cm]{};
    \node(b2)  [below of= a2, node distance=1cm]{};
    \node(b3) [below of= a3, node distance=1cm]{00};
    \node(b4) [below of= a4, node distance=1cm]{01};
    \node(b5) [below of= a5, node distance=1cm]{10};
    \node(b6) [below of= a6, node distance=1cm]{11};
    \node(b7) [below of= a7, node distance=1cm]{};
    \node(b8) [below of= a8, node distance=1cm]{};

    \node (c1) [below of= b1, node distance=1cm]{};
    \node(c2)  [below of= b2, node distance=1cm]{};
    \node(c3) [below of= b3, node distance=1cm]{001};
    \node(c4) [below of= b4, node distance=1cm]{010};
    \node(c5) [below of= b5, node distance=1cm]{101};
    \node(c6) [below of= b6, node distance=1cm]{110};
    \node(c7) [below of= b7, node distance=1cm]{};
    \node(c8) [below of= b8, node distance=1cm]{};

    \node (d1) [below of= c1, node distance=1cm]{0011};
    \node(d2)  [below of= c2, node distance=1cm]{0010};
    \node(d3) [below of= c3, node distance=1cm]{0101};
    \node(d4) [below of= c4, node distance=1cm]{0100};
    \node(d5) [below of= c5, node distance=1cm]{1011};
    \node(d6) [below of= c6, node distance=1cm]{1010};
    \node(d7) [below of= c7, node distance=1cm]{1101};
    \node(d8) [below of= c8, node distance=1cm]{1100};

    \node (e1) [below of= d1, node distance=1cm]{00110};
    \node(e2)  [below of= d2, node distance=1cm]{};
    \node(e3) [below of= d3, node distance=1cm]{};
    \node(e4) [below of= d4, node distance=1cm]{01001};
    \node(e5) [below of= d5, node distance=1cm]{10110};
    \node(e6) [below of= d6, node distance=1cm]{};
    \node(e7) [below of= d7, node distance=1cm]{};
    \node(e8) [below of= d8, node distance=1cm]{11001};

    \node (f1) [below of= e1, node distance=1cm]{001100};
    \node(f2)  [below of= e2, node distance=1cm]{001101};
    \node(f3) [below of= e3, node distance=1cm]{010010};
    \node(f4) [below of= e4, node distance=1cm]{010011};
    \node(f5) [below of= e5, node distance=1cm]{101100};
    \node(f6) [below of= e6, node distance=1cm]{101101};
    \node(f7) [below of= e7, node distance=1cm]{110010};
    \node(f8) [below of= e8, node distance=1cm]{110011};

    \node(g1) [below of= f1, node distance=1.5cm]{0011001};
    \node(g2)  [below of= f2, node distance=1.5cm]{0011010};
    \node(g3) [below of= f3, node distance=1.5cm]{0100101};
    \node(g4) [below of= f4, node distance=1.5cm]{0100110};
    \node(g5) [below of= f5, node distance=1.5cm]{1011001};
    \node(g6) [below of= f6, node distance=1.5cm]{1011010};
    \node(g7) [below of= f7, node distance=1.5cm]{1100101};
    \node(g8) [below of= f8, node distance=1.5cm]{1100110};

    \node(h1) [below of=g1, node distance=1cm]{};
    \node(h2)  [below of= g2, node distance=1cm]{00110100};
    \node(h3) [below of= g3, node distance=1cm]{01001011};
    \node(h4) [below of=g4, node distance=1cm]{};
    \node(h5) [below of=g5, node distance=1cm]{};
    \node(h6) [below of= g6, node distance=1cm]{10110100};
    \node(h7) [below of= g7, node distance=1cm]{11001011};
     \node(h8) [below of=g8, node distance=1cm]{};

    \node(i2)  [below of= h2, node distance=1cm]{\vdots};
    \node(i3) [below of= h3, node distance=1cm]{\vdots};
    \node(i6) [below of= h6, node distance=1cm]{\vdots};
    \node(i7) [below of= h7, node distance=1cm]{\vdots};

    \path[->] (a4) edge node {} (b3);
    \path[->] (a4) edge node {} (b4);
    \path[->] (a5) edge node {} (b5);
    \path[->] (a5) edge node {} (b6);

     \path[->] (b3) edge node {} (c3);
     \path[->] (b4) edge node {} (c4);
     \path[->] (b5) edge node {} (c5);
     \path[->] (b6) edge node {} (c6);

      \path[->] (b4) edge [bend right] node {} (b6);
      \path[->] (b5) edge [bend left] node {} (b3);

     \path[->] (c3) edge node {} (d1);
     \path[->] (c3) edge node {} (d2);
     \path[->] (c4) edge node {} (d3);
     \path[->] (c4) edge node {} (d4);
     \path[->] (c5) edge node {} (d5);
     \path[->] (c5) edge node {} (d6);
     \path[->] (c6) edge node {} (d7);
     \path[->] (c6) edge node {} (d8);

     \path[->] (d2) edge [bend right] node {} (d3);
     \path[->] (d3) edge [bend right] node {} (d5);

     \path[->] (d7) edge [bend left] node {} (d6);
     \path[->] (d6) edge [bend left] node {} (d4);

     \path[->] (d1) edge node {} (e1);
     \path[->] (d4) edge node {} (e4);
     \path[->] (d5) edge node {} (e5);
     \path[->] (d8) edge node {} (e8);

     \path[->] (e1) edge node {} (f1);
     \path[->] (e1) edge node {} (f2);
     \path[->] (e4) edge node {} (f3);
     \path[->] (e4) edge node {} (f4);
     \path[->] (e5) edge node {} (f5);
     \path[->] (e5) edge node {} (f6);
     \path[->] (e8) edge node {} (f7);
     \path[->] (e8) edge node {} (f8);

     \path[->] (f1) edge node {} (g1);
     \path[->] (f2) edge node {} (g2);
     \path[->] (f3) edge node {} (g3);
     \path[->] (f4) edge node {} (g4);
     \path[->] (f5) edge node {} (g5);
     \path[->] (f6) edge node {} (g6);
     \path[->] (f7) edge node {} (g7);
     \path[->] (f8) edge node {} (g8);

     \path[->] (g1) edge [out= 50, in=-130] node {} (f7);
     \path[->] (g2) edge node {} (h2);
     \path[->] (g3) edge node {} (h3);
     \path[->] (g4) edge [out = 150, in=-80] node {} (f1);
     \path[->] (g5) edge [out = 30, in=-100] node {} (f8);
     \path[->] (g6) edge node {} (h6);
     \path[->] (g7) edge node {} (h7);
     \path[->] (g8) edge [out=130, in=-50] node {} (f2);

     \path[->] (h2) edge node {} (i2);
     \path[->] (h3) edge node {} (i3);
     \path[->] (h6) edge node {} (i6);
     \path[->] (h7) edge node {} (i7);

}
\end{tikzpicture}
%}
 \caption{The HB diagram of the Morse minimal subshift.} \label{fig:MMS}
\end{center}
\end{figure}
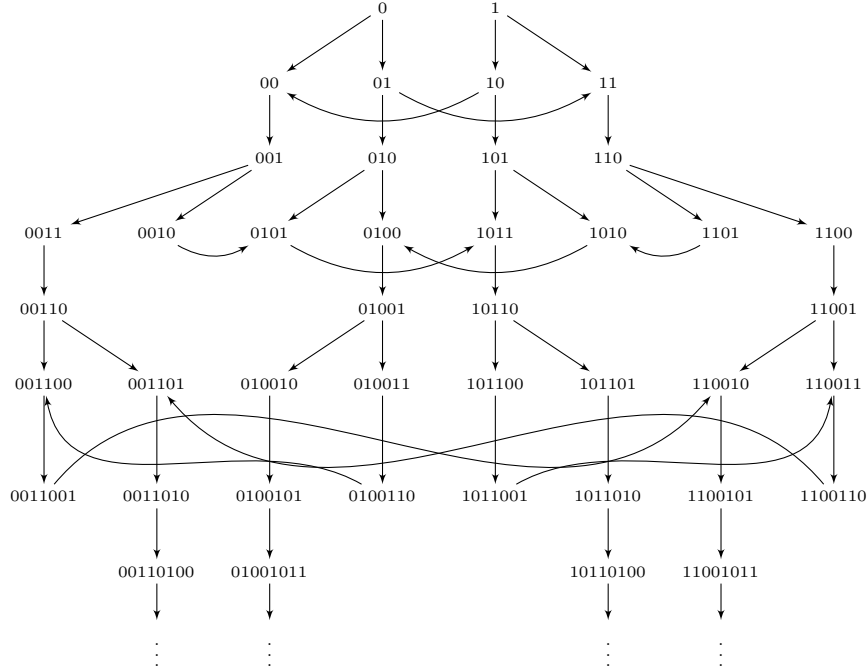

\section{Conclusion and Further Directions}
 We have described the construction of HB diagrams for some highly non-Markovian systems, that is, systems with long-range memory. These HB diagrams provide a way to visualize all the possibilities for extending any given block and {present} useful information about the languages of such systems and therefore about the structures of the systems themselves.
Here are a few questions about how the diagrams might be put to further use.
How can invariant measures be represented on the diagrams?  Can we detect unique ergodicity or minimality from these diagrams?
In Section \ref{Morse} we were able to construct the HB diagram of the Morse minimal subshift because the recognizability property of the Morse substitution allowed us to say precisely which blocks are significant. Can this result be generalized to any recognizable, or bilaterally recognizable, substitution?
It is known that beta shifts, as well as their factors, have unique measures of maximal entropy \cite{Hofbauer,Walters79,Climenhaga}.
{The HB diagram of a $\beta$-shift turns out to be just a relabeling of the well-known $\beta$-shift graph (see, e.g., \cite{Johnson,Chile}).}
Is there a simple way to transform the HB diagram of a subshift to produce the HB diagram of one of its factors? {A} relation between the two diagrams {could} help to understand factor maps, and in particular to identify measures of maximal entropy or maximal relative entropy.

\begin{ack*}
This paper is based on the UNC-Chapel Hill master's project of the first author, written under the direction of the second author.
\end{ack*}

\bibliographystyle{amsplain}
\bibliography{New}

\providecommand{\bysame}{\leavevmode\hbox to3em{\hrulefill}\thinspace}
\providecommand{\MR}{\relax\ifhmode\unskip\space\fi MR }
\providecommand{\MRhref}[2]{%
  \href{http://www.ams.org/mathscinet-getitem?mr=#1}{#2}
}
\providecommand{\href}[2]{#2}
\begin{thebibliography}{10}

\bibitem{Buzzi1}
J{\'e}r{\^o}me Buzzi, \emph{Intrinsic ergodicity of smooth interval maps},
  Israel J. Math. \textbf{100} (1997), 125--161. \MR{1469107 (99g:58071)}

\bibitem{Buzzi}
\bysame, \emph{Entropy theory on the interval}, \'{E}cole de {T}h\'eorie
  {E}rgodique, S\'emin. Congr., vol.~20, Soc. Math. France, Paris, 2010,
  pp.~39--82. \MR{2856509 (2012j:37010)}

\bibitem{Buzzi-Hubert}
J{\'e}r{\^o}me Buzzi and Pascal Hubert, \emph{Piecewise monotone maps without
  periodic points: rigidity, measures and complexity}, Ergodic Theory Dynam.
  Systems \textbf{24} (2004), no.~2, 383--405. \MR{2054049 (2005a:37063)}

\bibitem{Climenhaga}
Vaughn Climenhaga and Daniel~J. Thompson, \emph{Intrinsic ergodicity beyond
  specification: {$\beta$}-shifts, {$S$}-gap shifts, and their factors}, Israel
  J. Math. \textbf{192} (2012), 785--817. \MR{3009742}

\bibitem{Coven}
Ethan~M. Coven, Michael Keane, and Michelle Lemasurier, \emph{A
  characterization of the {M}orse minimal set up to topological conjugacy},
  Ergodic Theory Dynam. Systems \textbf{28} (2008), no.~5, 1443--1451.
  \MR{2449536 (2010c:37017)}

\bibitem{Fiebig}
Doris Fiebig and Ulf-Rainer Fiebig, \emph{Covers for coded systems}, Symbolic
  {D}ynamics and {I}ts {A}pplications ({N}ew {H}aven, {CT}, 1991), Contemp.
  Math., vol. 135, Amer. Math. Soc., Providence, RI, 1992, pp.~139--179.
  \MR{1185086 (93m:54068)}

\bibitem{Fife}
Earl~D. Fife, \emph{Binary sequences which contain no {$BBb$}}, Trans. Amer.
  Math. Soc. \textbf{261} (1980), no.~1, 115--136. \MR{576867 (82a:05034)}

\bibitem{Fischer}
Roland Fischer, \emph{Sofic systems and graphs}, Monatsh. Math. \textbf{80}
  (1975), no.~3, 179--186. \MR{0407235 (53 \#11018)}

\bibitem{Fogg}
N.~Pytheas Fogg, \emph{Substitutions in {D}ynamics,{A}arithmetics and
  {C}ombinatorics}, Lecture Notes in Mathematics, vol. 1794, Springer-Verlag,
  Berlin, 2002, Edited by V. Berth{\'e}, S. Ferenczi, C. Mauduit and A. Siegel.
  \MR{1970385 (2004c:37005)}

\bibitem{OnlynoBBb}
W.~H. Gottschalk and G.~A. Hedlund, \emph{A characterization of the {Morse}
  minimal set}, Proc. Amer. Math. Soc. \textbf{15} (1964), 70--74. \MR{0158386
  (28 \#1609)}

\bibitem{Hofbauer}
Franz Hofbauer, \emph{On intrinsic ergodicity of piecewise monotonic
  transformations with positive entropy}, Israel J. Math. \textbf{34} (1979),
  no.~3, 213--237 (1980). \MR{570882 (82c:28039a)}

\bibitem{Hofbauer-Raith}
Franz Hofbauer and Peter Raith, \emph{Topologically transitive subsets of
  piecewise monotonic maps, which contain no periodic points}, Monatsh. Math.
  \textbf{107} (1989), no.~3, 217--239. \MR{1008681 (91c:58074)}

\bibitem{Johnson}
Kimberly~Christian Johnson, \emph{Beta-shift {D}ynamical {S}ystems and their
  {A}ssociated {L}anguages}, ProQuest LLC, Ann Arbor, MI, 1999, Thesis
  (Ph.D.)--The University of North Carolina at Chapel Hill. \MR{2699401}

\bibitem{Krieger}
Wolfgang Krieger, \emph{On sofic systems. {I}}, Israel J. Math. \textbf{48}
  (1984), no.~4, 305--330. \MR{776312 (86j:54074)}

\bibitem{Krieger2}
\bysame, \emph{On sofic systems. {II}}, Israel J. Math. \textbf{60} (1987),
  no.~2, 167--176. \MR{931874 (89k:54098)}

\bibitem{LM}
Douglas Lind and Brian Marcus, \emph{An {I}ntroduction to {S}ymbolic {D}ynamics
  and {C}oding}, Cambridge University Press, Cambridge, 1995. \MR{1369092
  (97a:58050)}

\bibitem{Lothaire}
M.~Lothaire, \emph{Algebraic {C}ombinatorics on {W}ords}, Encyclopedia of
  Mathematics and its Applications, vol.~90, Cambridge University Press,
  Cambridge, 2002, A collective work by Jean Berstel, Dominique Perrin, Patrice
  Seebold, Julien Cassaigne, Aldo De Luca, Steffano Varricchio, Alain Lascoux,
  Bernard Leclerc, Jean-Yves Thibon, Veronique Bruyere, Christiane Frougny,
  Filippo Mignosi, Antonio Restivo, Christophe Reutenauer, Dominique Foata,
  Guo-Niu Han, Jacques Desarmenien, Volker Diekert, Tero Harju, Juhani
  Karhumaki and Wojciech Plandowski, With a preface by Berstel and Perrin.
  \MR{1905123 (2003i:68115)}

\bibitem{Unendingchess}
Marston Morse and Gustav~A. Hedlund, \emph{Unending chess, symbolic dynamics
  and a problem in semigroups}, Duke Math. J. \textbf{11} (1944), 1--7.
  \MR{0009788 (5,202e)}

\bibitem{Mosse}
Brigitte Moss{\'e}, \emph{Puissances de mots et reconnaissabilit\'e des points
  fixes d'une substitution}, Theoret. Comput. Sci. \textbf{99} (1992), no.~2,
  327--334. \MR{1168468 (93f:68076)}

\bibitem{KEP}
Karl Petersen, \emph{Ergodic {T}heory}, Cambridge Studies in Advanced
  Mathematics, vol.~2, Cambridge University Press, Cambridge, 1989, Corrected
  reprint of the 1983 original. \MR{1073173 (92c:28010)}

\bibitem{Chile}
\bysame, \emph{Information compression and retention in dynamical processes},
  Dynamics and {R}andomness ({S}antiago, 2000), Nonlinear Phenom. Complex
  Systems, vol.~7, Kluwer Acad. Publ., Dordrecht, 2002, pp.~147--217.
  \MR{1975578 (2005d:37020)}

\bibitem{Quef}
Martine Queff{\'e}lec, \emph{Substitution {D}ynamical {S}ystems---{S}pectral
  {A}nalysis}, second ed., Lecture Notes in Mathematics, vol. 1294,
  Springer-Verlag, Berlin, 2010. \MR{2590264 (2011b:37018)}

\bibitem{Walters79}
Peter Walters, \emph{Equilibrium states for {$\beta $}-transformations and
  related transformations}, Math. Z. \textbf{159} (1978), no.~1, 65--88.
  \MR{0466492 (57 \#6370)}

\bibitem{Weiss}
Benjamin Weiss, \emph{Subshifts of finite type and sofic systems}, Monatsh.
  Math. \textbf{77} (1973), 462--474. \MR{0340556 (49 \#5308)}

\end{thebibliography}

\end{document}